\documentclass[11pt]{article}
 \usepackage{PRIMEarxiv}
\usepackage[numbers]{natbib}
\usepackage[utf8]{inputenc} 
\usepackage[T1]{fontenc}    
\usepackage{url}            
\usepackage{booktabs}       
\usepackage{amsfonts}       
\usepackage{nicefrac}       
\usepackage{microtype}      
\usepackage{lipsum}
\usepackage{fancyhdr}       
\usepackage{graphicx}       
\graphicspath{{media/}}     
\usepackage{amssymb, amsmath,, amsthm,color,url}
\usepackage[normalem]{ulem}
\usepackage{dsfont}
\usepackage{scrextend}
\usepackage{mathtools}
\usepackage{enumitem}

\newtheorem{theorem}{Theorem}[section]
\newtheorem{proposition}[theorem]{Proposition}
\newtheorem{lemma}[theorem]{Lemma}
\newtheorem{example}[theorem]{Example}
\newtheorem{corollary}[theorem]{Corollary}

\newtheorem{definition}[theorem]{Definition}
\newtheorem{observation}[theorem]{Observation}

\makeatletter
\newtheorem*{rep@theorem}{\rep@title}
\newcommand{\newreptheorem}[2]{%
\newenvironment{rep#1}[1]{%
 \def\rep@title{#2 \ref{##1}}%
 \begin{rep@theorem}}%
 {\end{rep@theorem}}}
\makeatother

\newreptheorem{theorem}{Theorem}
\newreptheorem{proposition}{Proposition}
\newreptheorem{lemma}{Lemma}
\newreptheorem{example}{Example}
\newreptheorem{corollary}{Corollary}
\newreptheorem{remark}{Remark}
\newreptheorem{definition}{Definition}
\newreptheorem{observation}{Observation}

\def\od{\stackrel{\mathrm{def}}{=}}

\newcommand{\ZZ}{\mathbb{Z}}

\newcommand{\R}{\mathbb{R}}
\newcommand{\X}{\mathcal{X}}

\newcommand{\1}{\mathds{1}}

\def\Spec{{\text{\upshape Spec}}}
\DeclareMathOperator{\im}{im}

\newcommand{\Lup}{L^{\mathrm{up}}} 
\newcommand{\Ldown}{L^{\mathrm{down}}} 

\newcommand{\Span}{\mathrm{Span}}
\def\up{{\text{\upshape up}}}
\def\down{{\text{\upshape down}}}
\def\d{{\delta}} 

\newcommand{\glue}[3]{{#1}^{+_{#2}:{#3}}}
\newcommand{\remove}[2]{{#1}^{\times_{#2}}}
\newcommand{\swap}[3]{{#1}^{\mathrlap{\times}{+}_{#2}:{#3}}}

\usepackage[dvipsnames]{xcolor}

\usepackage{tikz}
\usepackage{pgfplots}
\usepackage{adjustbox}
\usepackage{subcaption}

 \usetikzlibrary{calc}
 \usetikzlibrary{fit}
 \usetikzlibrary{matrix}
 \usetikzlibrary{patterns}
 
\usetikzlibrary{shapes.multipart,positioning}

\usepackage{xcolor}
\colorlet{myPurple}{Orchid!80!black}
\colorlet{myBlue}{Turquoise!60!black}
\colorlet{lyellow}{yellow!20!white}
\colorlet{myGreen}{green!50!black}
\colorlet{lgray}{gray!15!white}
\colorlet{lblue}{Turquoise!20!white}
\colorlet{lgreen}{ForestGreen!15!white}
\colorlet{lgrey}{gray!15!white}
\colorlet{dblue}{Turquoise!60!black}
\colorlet{dgreen}{ForestGreen!60!black}

\usetikzlibrary{arrows, decorations.markings}
      \tikzstyle{vecArrow} = [thick, decoration={markings,mark=at position
      1 with {\arrow[semithick]{open triangle 60}}},
      double distance=1.4pt, shorten >= 5.5pt,
      preaction = {decorate},
      postaction = {draw,line width=1.4pt, white,shorten >= 4.5pt}]
     \tikzstyle{innerWhite} = [semithick, white,line width=1.4pt, shorten >= 4.5pt]
\usetikzlibrary{positioning}

\usetikzlibrary {patterns.meta}
\tikzdeclarepattern{
  name=mylines,
  parameters={
      \pgfkeysvalueof{/pgf/pattern keys/size},
      \pgfkeysvalueof{/pgf/pattern keys/angle},
      \pgfkeysvalueof{/pgf/pattern keys/line width},
  },
  bounding box={
    (0,-0.5*\pgfkeysvalueof{/pgf/pattern keys/line width}) and
    (\pgfkeysvalueof{/pgf/pattern keys/size},
     0.5*\pgfkeysvalueof{/pgf/pattern keys/line width})},
  tile size={(\pgfkeysvalueof{/pgf/pattern keys/size},
              \pgfkeysvalueof{/pgf/pattern keys/size})},
  tile transformation={rotate=\pgfkeysvalueof{/pgf/pattern keys/angle}},
  defaults={
    size/.initial=5pt,
    angle/.initial=45,
    line width/.initial=.4pt,
}, code={
\draw [line width=\pgfkeysvalueof{/pgf/pattern keys/line width}] (0,0) -- (\pgfkeysvalueof{/pgf/pattern keys/size},0);
}, }

\title{On the spectrum of the Hodge Laplacian on sequences
}

\author{
 Hannah Santa Cruz Baur \\
  Pennsylvania State University \\
  \texttt{hqs5441@psu.edu} \\
   \And
  Vladimir Itskov \\
  Pennsylvania State University \\
  \texttt{vladimir.itskov@psu.edu} \\
}

\makeatletter
\begin{document}
\maketitle

\begin{abstract}
Hodge Laplacians have been previously proposed as a natural tool for understanding higher-order interactions in networks and directed graphs. Here we introduce a Hodge-theoretic approach to spectral theory and dimensionality reduction for probability distributions on sequences and simplicial complexes. We demonstrate that this Hodge theory has desirable properties with respect to  natural null-models, where the underlying vertices are independent. 

We prove that for the case of independent vertices in simplicial complexes, the appropriate Laplacians are multiples of the identity and thus have no meaningful Fourier modes. For the null model of   independent vertices in sequences, we prove that the  appropriate Hodge Laplacian has an integer spectrum, and describe its eigenspaces. We also prove that  the underlying cell complex of sequences has trivial reduced homology.   Our results establish a foundation for developing  Fourier analyses of probabilistic models, which are common in theoretical neuroscience and machine-learning.

\end{abstract}


\section{Introduction}
Graph Laplacians and diffusion maps \cite{Kirchhoff1847} are a common tool in network analysis, dimensionality reduction and deep learning networks \cite{Belkin2003, coifman2005geometric, Bronstein2017}.  In these settings a pairwise similarity function on a finite set of vertices is converted into a weighted graph;  then a Laplacian is defined over the space of vertices so that the corresponding heat equation describes the diffusion process on the weighted graph \cite{Ricaud2019}.  The eigenfunctions of the  Laplacian  are understood as an analog of  the standard Fourier basis in  a Euclidean domain,  and they are often used to embed the vertices of the graph into a low-dimensional space. The resulting geometric representation accurately reflects the similarity function, interpreted as a proxy for a distance \cite{Belkin2003,coifman2005geometric}.  

Pairwise correlations are  often used for characterizing  similarity on a vertex set. However, in many application domains, such as theoretical neuroscience and machine learning,  higher-order correlations play an important role, and   essential  features of the underlying data cannot be detected using {\it only}  second-order correlations.  Hodge Laplacians for simplicial complexes, dating back to \cite{Eckmann1944},  have been proposed as a proper  replacement for the graph Laplacians. 
Basic  properties of the Hodge Laplacians with respect to the structure of the simplicial complexes were investigated in the pioneering work \cite{Horak2013,Horak2011InterlacingIF}, 
and later generalized to the context of  cellular sheaves in \cite{Hansen_2019,ghrist2020cellular}.
Similar to the graph Laplacians, the diffusion equation for the Hodge 
Laplacians was reinterpreted in terms of random walks on simplicial complexes \cite{mukherjee2016random,Schaub2020,kaufman2020high}.  Recently, there has been a resurgence of interest in Hodge-theoretic tools for generalizing various deep  learning networks \cite{Ebli2020,Barbarossa_2020,Bodnar2021,roddenberry2021principled,giusti2022simplicial,keros2022dist2cycle,roddenberry2022signal,bodnar2022neural}\footnote{This list of papers is very incomplete.}. 

While Hodge Laplacians are well-suited for a spectral theory on binary vectors (or subsets), many  scientific fields, such as genomics, theoretical neuroscience and language models study data that come in the form of a collection of sequences. This kind of data can be thought of as a  generalization of directed graphs, similar to how simplicial complexes and hypergraphs generalize undirected graphs.  
The sequence data can be formalized in terms of a cell complex $\X$, that comprises all possible sequences, and a probability distribution 
 $p : \X \to [ 0,1 ] $ that  describes the statistical properties of the data. Here, the role of similarity, or correlations is played by the probability;  an equivalent formulation can be made in terms of the higher-order correlations.  
 
Similar to the case of binary vectors, the interpretation and machine-learning pipelines for sequential data  require   tools such as dimensionality reduction and the notion of convolution. This necessitates a   Laplacian formalism for sequences. Importantly, such a  Laplacian needs to have  particularly simple properties  in the case of {\it null models}, where the statistical properties of sequences can be explained by a simple probabilistic model with no other  underlying structure.

To this end we developed a Hodge Laplacian formalism associated with probability distributions on sequence complexes. We accompany this construction with an analogous one for the case of simplicial complexes. We  consider natural null models of {\it independent vertices} and prove that for these models the Hodge Laplacians have particularly simple properties.

The paper is organized as follows. In section \ref{sec:hodge} we provide  relevant background on abstract cell complexes, then  we  introduce weight functions associated to probability distributions and define  the null models  of independent vertices (subsections \ref{subsec:dist}, \ref{subsec:ind}); we end the section with an exposition of classical definitions and results in  Hodge theory (section \ref{subsec:hodge}) in the context of abstract cell complexes. In section \ref{results} we present the main results, that describe the Laplacian spectrum for independent vertices models, which we prove in sections \ref{sec:seq cpx} and \ref{sec:Simpl cpx} for sequence and simplicial complexes respectively. 
 
\medskip
\section{Hodge theory on abstract cell complexes}\label{sec:hodge}

\subsection{Abstract Cell Complexes} \label{subsec:cpx}
 Sequence complexes are a special type of abstract cell complexes; thus start with  their formal definition.\\

\begin{definition}\label{def:ACC}
An  \emph{Abstract Cell  Complex} is a quadruple $\X=(\X, \leq, \dim,\kappa)$ that consists of a  poset $(\X, \leq)$ along with two functions, the dimension  function $\operatorname{dim}: \X \rightarrow \mathbb{Z}$,   and the incidence  function $\kappa: \X \times \X \rightarrow \ZZ$ that  satisfy the following conditions:
\begingroup
\addtolength{\jot}{1em}
\begin{align*}
&1.  \hspace{1mm} \dim \text{ is a poset morphism;}\\
&2. \hspace{1mm} \text{For all } \xi, \xi^{\prime} \in \X,  \hspace{10mm}
\kappa\left(\xi, \xi^{\prime}\right) \neq 0 \quad  \implies \quad     \xi^\prime\leq  \xi, \text{ and } 
\operatorname{dim}(\xi)=\operatorname{dim}\left(\xi^{\prime}\right)+1 ;\\
&3. \hspace{1mm} \text{For all } \xi, \xi^{\prime \prime} \in \X,  \hspace{10mm}
\sum_{\xi^{\prime} \in \X} \kappa\left(\xi, \xi^{\prime}\right)\kappa\left(\xi^{\prime}, \xi^{\prime \prime}\right)=0
\end{align*}
\endgroup
 \end{definition}
 \medskip 
 Note that we will only consider abstract cell complexes  where $\X_n=\dim^{-1}(n)$ is a finite set for all $n \in \ZZ$.
 The most popular example of abstract cell complexes are simplicial complexes.\\

\begin{definition}\label{ex:simplicial}  Let $V$ be a finite set of vertices,
  an \emph{Abstract Simplicial Complex},  $K\subseteq 2^{V}$, is a collection of subsets of  $V$, that  is closed under inclusion, meaning that 
$$\forall \xi \in K,  \quad \xi' \subseteq \xi \implies \xi' \in K.$$
  \end{definition}

 Here the poset structure is chosen to be the set inclusion $\subseteq$, and   the dimension function is    $\dim(\xi)= \vert \xi\vert-1,$ 
where $\vert \xi\vert $ is the cardinality  of the set $\xi$. 
The incidence function is defined using a particular choice of a total order $v_1<v_2< \dots <v_m$ on the vertex set $V$. Given a set  $\xi=\{u_0,u_1,\dots,u_n\}$, where $u_j<u_{j+1}$ for all $j=0,\dots n-1$,  and a subset  $\xi' \subseteq  \xi$ such that $\dim(\xi')=\dim(\xi)-1$,
\begin{equation*}  \kappa\left(\xi, \xi^{\prime}\right) \od (-1)^j, \quad \text{where } \xi^\prime = \xi\setminus \{u_j\}. 
\end{equation*} 
If $\dim(\xi')\neq \dim(\xi)-1$, or $\xi'\not\subseteq \xi$, then $ \kappa\left(\xi, \xi^{\prime}\right)=0$.

\medskip
Another example of interest is a collection of sequences. Let $V$ be a finite set of vertices, denote by $\X_{n}[V]$ the set of all  length $(n+1)$ sequences with vertices in $V$. Here the same vertex can appear in the sequence more than once, and  sequences with different orderings are distinct. We say that a sequence   $\tau\in \X_{n-k}[V]$ is a subsequence  of a sequence  $\sigma\in \X_{n}[V]$, with the notation  $\tau\leq \sigma$,  if    $\tau$ can be obtained by removing $k$ vertices (in any position) from the sequence $\sigma$. For example, $(a,b,c)\leq (c,a,b,b,c,a)$. We also denote by $\X_{-1}=\{\varnothing\}$ the set consisting of the  empty sequence in $V$. \\

\begin{definition}\label{ex:seqcpx}  Let $V$ be  a finite set of vertices. A  \emph{Sequence Complex} is a collection of sequences 
$$\X = \bigcup_{n\geq -1} \X_n, \quad \text{ where } \X_n\subseteq \X_{n}[V],$$  
which is closed under inclusion, i.e. 
$\forall \sigma  \in \X,  \,\,  \tau \leq \sigma \implies \tau  \in \X$.
\end{definition} 
This is a special case of an abstract cell complex, where the  poset structure is the subsequence relation $\leq$, and the dimension is defined as   $\dim(\sigma)=n=\vert \sigma \vert-1$ for $\sigma \in \X_{n}[V]$. The incidence function $\kappa$  is defined for all $ \sigma, \tau \in \X,$ with $\tau \leq \sigma$ and $\dim(\tau)=\dim(\sigma)-1$, as 
\begin{equation*}
 \kappa(\sigma,\tau)\od \sum_{ i: \remove{\sigma}{i}=\tau }(-1)^i, 
\end{equation*}
where $\remove{\sigma}{i}$ denotes the sequence, obtained by removing the $i$'th vertex from sequence $\sigma$. 
If $\dim(\tau)\neq \dim(\sigma)-1$, or $ \tau \not \leq\sigma$  then $ \kappa\left(\sigma, \tau \right)=0$.\\

In this paper, we will only consider  $\X_n=\X_n[V]$ the full collection of sequences over the vertex set $V$. We refer to the corresponding sequence complex as the full sequence complex,  even though all the properties and results,  except for Theorem \ref{bigThm}, and its corollary  hold in the more general case.

\subsection{Distributions on Simplicial and Sequence  Complexes}\label{subsec:dist}

A probability distribution on a sequence complex  $\X$, is a function $p:  \X \to [0,1]$ that satisfies the condition 
$$ \sum_{\sigma \in \X}   p(\xi) =1.$$

We introduce  weights $w_p : \X \to [ 0,1 ] $, that  are 
 induced by the distribution $p$, 
 as the conditional probability given the sequence  length: 
\begin{equation} \label{eq:wpsigma} w_p(\sigma)=p(\sigma \bigm| | \sigma| =n+1)= {p(\sigma )}\left( \sum_{\vert \tau\vert = \vert \sigma\vert }p(\tau)\right)^{-1}, \qquad \sigma\in \X_n. \end{equation}

An immediate consequence of this weight function construction is that all sequences of same length have weights that sum to one.
\begin{equation}\label{weightsum}
\sum_{\sigma \in \X_n}w_p(\sigma)=1,  \hspace {2mm} \text{for all } n \in \mathbb{Z}.
\end{equation}

\medskip
In parallel to the distribution induced weight function on sequences, we introduced two weight functions on simplicial complexes. A simplicial distribution on an abstract simplicial complex $K \subseteq 2^{[m]}$, is a probability distribution $p: K \to [0,1]$ that satisfies the condition
\begin{align*}
 \sum_{\xi \in K} \hspace{1mm}&p(\xi) =1
\end{align*}

We consider either the simplicial distribution directly as a weight function on the simplicial complex, or the associated moment map as the weight function. 
The moment map $m_p: K \to [0,1]$, associated to a simplicial distribution, $p$ is defined for all $\xi \in K$ as:
\begin{align*}
m_p(\xi) := \mathbb{E}_p \Bigg( \prod_{i \in \xi} x_i \Bigg)
= \sum_{\zeta \supseteq \xi} p(\zeta)
\end{align*}

Here we use the notation $x_i : 2^{[m]} \to \{0,1\}, x_i(\xi)=1 \text{ iff } i \in \xi$, to refer to the indicator function, viewed as a random variable on $2^{[m]}$. The equivalence of the two definitions for the moment map is easy to see,
$$\mathbb{E}_p \Bigg( \prod_{i \in \xi} x_i \Bigg) 
= \sum_{\zeta \supseteq \xi} p(\zeta) \prod_{i \in \xi} x_i (\zeta) + \sum_{\zeta \not\supset \xi} p(\zeta)\prod_{i \in \xi} x_i (\zeta)
= \sum_{\zeta \supseteq \xi} p(\zeta)\cdot 1 +  \sum_{\zeta \not\supseteq \xi} p(\zeta) \cdot 0
= \sum_{\zeta \supseteq \xi} p(\zeta)
$$

\medskip
Notice the moment map, as opposed to the simplicial distribution, takes value one on the empty set, $m_p(\varnothing)=1$, and 
is monotone decreasing on the faces of the simplicial complex, 
$$\zeta \subseteq \xi \implies m_p(\zeta) > m_p(\xi).$$

\subsection{The independence models for simplices and sequences}\label{subsec:ind}
 We  define null models on sequence and simplicial complexes as weight functions arising  from an assumption that the weights assigned to the sequences (or simplices) depend solely on the vertices involved, and not the precise position of each vertex.
We will refer to complexes with such weight functions as following the Independent Vertices Model. Notice we would have one such model for the sequence complex context, and another for the simplicial complex context. We make this formal in the following two definitions.\\

\begin{definition}\label{def:ind seq}
Given a vertex set $V$,  and vector  $(w_v)_{v \in V} \in (0,1)^{|V|}$ which satisfies:
  \begingroup
\addtolength{\jot}{0.5em}
\begin{alignat}{2}\label{ind model}
\sum_{v\in V} w_v&=1
\end{alignat}
\endgroup
We define the associated  \emph{independent vertices model} over a sequences complex $\X$, with $\X_0=V$, as the weighted sequence complex $(\X,w)$ with weight function $w$,  such that:
  \begingroup
\addtolength{\jot}{0.5em}
\begin{alignat}{2}\label{ind model}
w((v_0,...,v_n))&= \prod_{i=0}^n w_{v_i}, \hspace{5mm} && \text{for all } (v_0,...,v_n) \in \X \nonumber\\
w(\varnothing) &= 1 \nonumber
\end{alignat}
\endgroup
\end{definition}

It is quick to check that this weight function satisfies the condition $$\sum_{\sigma \in \X_n} w(\sigma)=1 \quad \text{ for all } n \geq -1.$$   One can interpret this weight function as a probability  distribution on sequences of a fixed length $=n+1$, where the appearance of a vertex in each sequence position is independent from the vertices in all the other positions, and $w_v$ is the  probability of the  vertex $v$ appearing in each position. \\

In the case of simplicial complexes, we define the independent simplicial distribution associated to a vector $(p_1,...,p_m) \in (0,1)^{m}$ as follows. For each face $\xi \in 2^{[m]}$, interpreted as a binary vector on $m$ vertices, we define the probability $p(\xi)$, as the joint probability of each of the vertices $i \in [m]$ taking the prescribed binary value in $\xi$ following a Bernoulli distribution given by $p_i$. In other words, the probability is defined as:

\begingroup
\addtolength{\jot}{0.5em}
\begin{alignat*}{2}
p(\xi)&= \prod_{i \in \xi} p_i \prod_{j \not\in \xi} (1-p_j),  \hspace{5mm} &&\text{ for all } \xi \in K. 
\end{alignat*}
\endgroup

Here $p_i$ has the meaning of probability of the $i$-th vertex appearing in a subset $\sigma$. Note that these probabilities do {\it not}   sum to $1$. 
We define the independent vertices model on simplicial complexes, inspired in the moment map weight function associated to the independent simplicial distribution.\\

\begin{observation}\label{obs:moment}
The moment map $m_p$,  on a simplicial complex $K \subseteq 2^{[m]}$ induced by  the independent simplicial distribution associated to $(p_1,...,p_m) \in (0,1)^{m}$,  is defined as:
  \begingroup
\addtolength{\jot}{0.2em}
\begin{alignat*}{2}
m_p(\xi)&= \prod_{i \in \xi} p_i,  \hspace{5mm} &&\text{ for all } \xi \in K \\
w(\varnothing) &= 1 
\end{alignat*}
\endgroup
\end{observation}

We give the definition for the independent vertices model in the context of simplicial complexes, for a general vector $(w_1,...,w_m) \in (0,1)^{m}$ determining the weight of each vertex. The definition is thus very reminiscent of the one in the sequence complex context.\\

\begin{definition}\label{def:ind model}
Given a vertex set $[m]$,  and vector  $(w_1,...,w_m) \in (0,1)^{m}$ (we allow in this context the possibility of $\sum_{i \in [m]}w_i \neq 1$), we define the associated  \emph{independent vertices model} over a simplicial complex $K \subseteq 2^{[m]}$, as the weighted simplicial complex $(K,w)$ with weight function $w$,  such that:
  \begingroup
\addtolength{\jot}{0.5em}
\begin{alignat}{2}
w(\xi)&= \prod_{i\in \xi}^n w_{i}, \hspace{5mm} && \text{for all } \xi \in K \nonumber\\
w(\varnothing) &= 1 \nonumber
\end{alignat}
\endgroup
\end{definition}

 In applications, we may interpret this weight function as either the moment map associated to an independent simplicial distribution (as in observation \ref{obs:moment}), or alternatively as the weight function that normalizes the distribution by the probability of the empty set, as in the following observation \ref{obs:norm dist}.   We will show, that in this null hypothesis setting, both weight functions yield an equivalent result, which is what we would like to see.\\
 
\begin{observation}\label{obs:norm dist}
Given a simplicial complex $K$, and the independent simplicial distribution, $p$, associated to the vector $(p_1,...,p_m) \in (0,1)^{m}$, consider the weight function $p_\varnothing$ defined over each face as it's probability normalized by the probability of the empty set.
\begin{align*}
p_\varnothing(\xi)=\dfrac{p(\xi)}{p(\varnothing)}, \hspace{5mm} \text{ for all } \xi \in K \\
\end{align*}
In the case of the independent simplicial distribution, this yields 
\begin{align*}
p_\varnothing(\xi)= \displaystyle\prod_{i \in \xi} p_i \cdot \displaystyle\prod_{i \not\in \xi} (1-p_i) \cdot \Bigg(\displaystyle\prod_{i \in [m]} (1-p_i)\Bigg)^{-1} = \displaystyle\prod_{i \in \xi} p_i (1- p_i)^{-1}
\end{align*}
Consequently the weighted simplicial complex associated to this weight function, $(K,p_\varnothing)$,  follows an independent vertices model associated to the vector 
$\Big( p_1 (1- p_1)^{-1},..., p_m (1- p_m)^{-1} \Big) \in (0,1)^{m}$.
\end{observation}

\subsection{Homological algebra  of Abstract Cell Complexes} \label{subsec:hom}

Given an abstract cell complex $\X$,  we denote the set  of $n$-dimensional elements  as  $\X_n \od  \dim^{-1}(n)$, and identify the  space of $n$-dimensional cochains  $C^n(\X)$ with  $\R$-valued functions on $\X_n$:  
\begin{equation*}
C^n(\mathcal{X})= \left \{ f: \X_n\to \R \right\}. 
\end{equation*}
 Note that if $\X_n$ is empty  for some $n\in \mathbb Z$,  then by convention,  $C^n(\mathcal{X})=\{ 0\}$.     We will refer to the graded space of cochains as, 
\begin{equation} \label{eq:CstarX} C^*(\mathcal{X})= \bigoplus_{n=-1}^{+\infty}C^n(\mathcal{X})\end{equation} 
Note  that we  restrict  ourselves to working over the field of real numbers;  this is because we will make use of inner products on  $C^*(\mathcal{X})$.\\

 \medskip 
The coboundary operator $\delta_{n}: C^{n}(\X) \rightarrow C^{n+1}(\X)$, is defined  on    $f \in C^n(\X)$ via 
\begin{equation*}
(\delta_{n} f) (\xi)\od
\sum_{\xi^{\prime} \in \X_{n}}  \kappa \left(\xi, \xi^{\prime} \right)   f  \left(\xi^{\prime}\right).
\end{equation*}
 
 The coboundary operator   satisfies $\delta_{n+1} \delta_n=0$ for all $n$; this is because of the condition (3) in Definition \ref{def:ACC}. This property enables  the usual definition of cohomology 
 $H^n(\X) = \frac{ \ker \delta_n}{ \im \delta_{n-1}}$ of the chain complex 
 \begin{align*}
\{0\}{\longrightarrow} \hspace{1mm} C^{-1}(\X) \stackrel{\delta_{-1}}{\longrightarrow}  C^{0}(\X) \stackrel{\delta_{0}}{\longrightarrow} \ldots \stackrel{\delta_{n-1}}{\longrightarrow} C^{n}(\X)\stackrel{\delta_{n}}{\longrightarrow}  C^{n+1}(\X){\longrightarrow}  \ldots.
\end{align*}

  \subsection{Hodge Theory on Abstract Cell Complexes}\label{subsec:hodge} 
Here we  expose the Hodge theory on  abstract cell complexes. The definitions and results we present here are classical results for simplicial complexes, dating back to \cite{Eckmann1944} and expanded in e.g.  \cite{Horak2013, Hansen_2019}. We present them in the context of  abstract cell complexes, for sake of self-containment  and notation consistency. The results and proofs of this section are not new. 

Let   $\langle \cdot  , \cdot \rangle$ be an inner product on 
$C^*(\X)$, such that the direct sum decomposition \eqref {eq:CstarX}  is an orthogonal decomposition with respect to the inner product $\langle \cdot  , \cdot \rangle$. Denote by $\delta_n^*$ the adjoint of $\delta_n$, with respect to the chosen inner product  $\langle \cdot  , \cdot \rangle$.  That is, for every $f,g\in C^*(\X)$, 
\begin{equation*}\langle \delta_n^* f, g\rangle = \langle f,  \delta_n g\rangle.
\end{equation*}  

It is easy to see that $\delta_{n+1}\delta_n=0$ if and only if $\delta^*_{n}\delta^*_{n+1}=0$. Therefore there are two versions of the cohomology space,  one for $\delta_n$ and the other for $\delta_n^*$, and these two are isomorphic.  We introduce the Laplacian operator, which  is closely related to the cohomology,  $H^n(\X)$, of the coboundary operator $\d_n$.\\
 
  \begin{definition}\label{def::Laplacian} The Laplacian of a triple $(C^*(\X),\delta,\langle \cdot , \cdot\rangle)$ is an operator over $ C^n(\X) \to  C^n(\X)$,  for all $n\in \mathbb{Z}_{\geq -1}$. \\
    \begingroup
\addtolength{\jot}{0.8em}
 \begin{align*}   \text{defined as,  }L_n &=\Lup_n+\Ldown_n\\
                        \text{where,  } \Lup_n  &=\delta_n ^* \delta_n,\hspace{2mm}
                           \Ldown_n  =\delta_{n-1} \delta_{n-1} ^*.                     
 \end{align*} 
 \endgroup
  \end{definition} 

\medskip
  For the reader familiar with the combinatorial Laplacian, $L_{\text{comb}}$, we note this operator corresponds to the $0$-dimensional up-Laplacian associated to a weighted graph G, 
  $$L_{\text{comb}}=L_0^\up:C^0(G)\to C^0(G).$$ 
  We refer to Section \ref{sec:Simpl cpx} for a deeper discussion of this.

 \medskip 
As a direct result of  Definition \ref{def::Laplacian} and standard Linear Algebra results, one can easily prove the following properties.\\
  
    \begin{lemma}\label{lemma:Laplacians}
The Laplacians $\ \Lup$, $\ \Ldown$ and $\ L$ are self-adjoint and   positive semi-definite with respect to  the inner product $\langle\cdot,\cdot\rangle$. Moreover, the following identities hold for all $n\in \mathbb{Z}_{\geq -1}$: 
\begingroup
\addtolength{\jot}{0.8em}
 \begin{align*}
  \ker \ L_n = &\left(\ker \ \Lup_n\right) \cap \left( \ker \ \Ldown_n\right) =   (\ker\delta_{n-1}^*)\cap (\ker \delta_n).\\
   \im(\delta_n^*) =& \im(\delta_n^*\delta_n) \text{  and } \im(\delta_{n+1}) = \im(\delta_{n+1}\delta_{n+1}^*).
\end{align*} 
\endgroup
\end{lemma}

\medskip 
 As we mentioned earlier, the Laplacian is intimately related to the cohomology of the coboundary operator $\d_n$. This theorem dates back to 1944 \cite{Eckmann1944}, in the context of simplicial complexes, and is easily  extended to the case of  abstract cell complexes with no changes to the proof.\\
 
  \begin{theorem}\label{thm:Erckmann} The kernel of the Laplacian $L_n$ associated to an abstract cell complex $\X$,  referred to as the space of ``harmonic vectors'',  is isomorphic to the cohomology of $\delta_n$, for all $n\in \mathbb{Z}_{\geq -1}$:  
  $$ \ker L_n\simeq  H^n(\X).
  $$ 
  \end{theorem}  
  
 Theorem \ref{thm:Erckmann} is a specific case of a  more detailed theorem below, we refer to \cite{Horak2013} for a proof.  We use the following notations in what follows:
  $$ B^n_+\od \im \delta_{n}^*,
	\quad 
	B^n_-\od \im \delta_{n-1}.$$

\medskip
  \begin{theorem}\label{thm:Hodge} Given an abstract cell complex, $\X$,  and a fixed $n \in \ZZ_{\geq -1}$, the space $C^n(\X)$ admits  the following $\langle \cdot ,\cdot \rangle$-orthogonal decomposition,: 
  \begin{equation*}  C^n(\X)=\ker L_n \oplus B^n_-\oplus B^n_+
  \end{equation*} 
   \end{theorem}  
 
Theorem \ref{thm:Hodge}  Has many nice spectral interpretations, let us introduce some notation so we may discuss them.  Denote the multiset (set with multiplicity) of strictly positive eigenvalues of $L_n $, $\Lup_n$ and $\Ldown_n$,  by $\Spec_*(L_n)$, $\Spec_*(\Lup_n)$ and $\Spec_*(\Ldown_n)$, respectively.   For each $\lambda\in\Spec_*(L_n)$, $\Spec_*(\Lup_n)$ and $\Spec_*(\Ldown_n)$,  denote the corresponding eigenspace by $E^n_\lambda$, $E_\lambda^{n,+}$ and $E_\lambda^{n,-}$, respectively. \\

    \begin{corollary}\label{cor:hodge}
   Given an abstract cell complex $\X$,  we may decompose the spectrum of it's associated Laplacian $L_n$, for any $n \in \ZZ_{\geq-1}$,  as:
   $$\Spec_*(L_n) = \Spec_*(L^\up_n)\sqcup\Spec_*(L^\down_n)= \Spec_*(L^\up_{n-1})\sqcup\Spec_*(L^\down_{n+1})$$
   Moreover, the eigenspaces associated to it's eigenvalues may be described as:
   \begin{align*}
   E_\lambda^{n,+} = E^n_\lambda\cap B^n_+ 
   \hspace{2mm}  \text{ and } \hspace{2mm}  
   E_\mu^{n,-} = E^n_\mu\cap B^n_- , 
   \hspace{5mm} \text{for } 
   \lambda\in\Spec_*(L^\up_n), \mu\in\Spec_*(L^\down_n)
   \end{align*}
   And isomorphically mapped onto the eigenspaces of the adjacent dimension Laplacian by
   $E_\lambda^{n,+} \underset{\d^*}{\overset{\d}{\rightleftarrows}} E_\lambda^{n+1,-}$.
\end{corollary} 

The proof to the corollary is a straightforward linear algebra exercise.

\medskip
\section{The main results}\label{results}
Here we present theorems that completely describe the spectrum of the Laplacians associated to weighted sequence and simplicial complexes under the independent vertices model. We provide  the proofs of these theorems in  Sections \ref{sec:seq cpx} and  \ref{sec:Simpl cpx}.\\

\begin{theorem}[\textbf{abridged}]\label{bigThm}  
Consider a weighted full sequence complex $(\X,w)$, with the weights of  the independent vertices model, then for any $n\geq 0$,  it's $n$-Laplacian, \, $L_n :  C^n(\X) \to C^n(\X)$ \, has the following eigenvalues: 
\begin{align*}
\lambda= \hspace{1mm}  &1,...,n+2\\
\text{with multiplicities} \hspace{53mm}&\\
\text{mult}(\lambda)= \hspace{1mm}  &\binom{n+1}{\lambda-1} \cdot (| \X_0 |-1)^{\lambda-1}. \hspace{50mm}
\end{align*}
\end{theorem}
We postpone  the proof until   Section \ref{sec:seq cpx}, where we give a more complete  theorem statement, which includes  explicit descriptions  of the associated eigenspaces. Let us describe here the eigenspaces for the smallest eigenvalues.\\

\begin{example}
The Laplacians $L_n$, of dimension $n \geq 0$, associated to 
a weighted full sequence complex $(\X,w)$  with weights following the independent vertices model, admit eigenvalues $\lambda= 1 \hspace{1mm} \& \hspace{1mm} 2$. Furthermore, by fixing $a \in \X_0$, we may describe the associated eigenspaces as:
\begin{align*}
    E(1,L_n)&= \Span \left\{ \sum_{\sigma \in \X_n } e_\sigma \right\},\\
    E(2,L_n)&= \Span \left\{ 
    \sum_{\substack{
    \sigma \in \X_n : \\ \sigma_i=a
    }} w(x) \cdot e_\sigma 
    -\sum_{\substack{
    \tau \in \X_n : \\ \tau_i=x
    }} w(a) \cdot e_\tau 
    \hspace{1mm} \Bigg| \hspace{2mm}
    i \in [0:n], \hspace{1mm}
    x \in \X_0 \setminus \{a\}
    \right\}.
\end{align*}
\end{example}
This example follows immediately from the full version of Theorem \ref{bigThm} given in Section \ref{sec:seq cpx}.\\

A simple corollary of our theorem, is that the cohomology of the full sequence complex is trivial.\\
\begin{corollary}
Given the full sequence complex $\X$, the cohomology is trivial, $H^n(\X)=\{0\}$,  for all $n \geq -1$.
\end{corollary}

\begin{proof}
    Fix $n \in \ZZ_{\geq -1}$, let $L_n$ be the Laplacian associated to the $(\X,w)$ with $w$ an independent vertices model weight function on $\X$. From Theorem \ref{thm:Erckmann}, we know 
    $H^n(\X) \simeq kerL_n$. 
    Applying Theorem \ref{bigThm}, which tells us $L_n$ does not admit the $0$ eigenvalue, we conclude:
    $$H^n(\X) \simeq kerL_n=\{0\}.$$
\end{proof}

Here we should note that the homology of the {\it complex of injective words} was computed previously in \cite{Farmer78,BjornerWachs1983}, and is equal to the homology of a bouquet of spheres.  The complex of injective words is a special case of a sequence complex (see Definition \ref{ex:seqcpx}),   where each vertex is allowed to appear only once in a sequence. The full complex of sequence that we consider is the maximal possible complex on a vertex set $V$, thus it is perhaps {\it expected} that it should have a trivial homology/cohomology.

\medskip
We contrast the Laplacian spectrum of the sequence complex under the independent vertices model, with the  Laplacian associated to a simplicial complex under the same model.  Interestingly, here the eigenfunctions reveal no interesting structure. Although our weight functions differ sightly in the different contexts, we cannot explain why the beautiful integer spectrum appears in the sequence setting, but not the simplicial.\\

\newpage
\begin{theorem}\label{prop:IndSimpLap}
Let $(2^{[m]},w)$ be a weighted simplicial complex, with  weight function $w:2^{[m]} \to \R_{>0}$ such that $w(\varnothing)=1$.  Then the following two conditions are equivalent:
\begin{enumerate}
\item[(i)] $(2^{[m]},w)$ is an independent vertices model;
\item[(ii)] the associated Laplacians are multiples of identities:
\begin{equation} \label{eq: multiple of identity}
L_n=\alpha_n   I_{C^n}, \quad \forall n\geq -1,   
\end{equation}
where $I_{C^n}$ denotes  the identity on $C^n(2^{[m]})$, and $\alpha_n \in \R$.
\end{enumerate} 

Moreover, if \eqref{eq: multiple of identity} holds, then  $\alpha_n$ is constant with respect to  $n$, and  
\begin{equation*} \alpha_n = \sum_{i=1}^m w_i,\end{equation*}
where  $(w_1,\dots,w_m)$ is the vector determining the independent vertices model defining the weight function $w$.\\
\end{theorem}
We give the proof to Theorem \ref{prop:IndSimpLap} in Section \ref{sec:Simpl cpx}. Note that this proof can be seen as a natural consequence of the results in  \cite{Horak2013}.
We can  apply this theorem  to the aforementioned moment map and simplicial distribution weight functions associated to the independent simplicial distribution.

 \begin{corollary}\label{cor:simplicial} 
Given the simplicial complex $2^{[m]}$, and the independent simplicial distribution, $p$, associated to the vector $(p_1,...,p_m) \in (0,1)^{m}$, the following two hold:
\begin{enumerate}
    \item [(i)] The Laplacian associated to the moment map weighted simplicial complex $(K,m_p)$, is  
    $$L_n = \left( \sum_{i=1}^m p_i \right)  I_{C^n}, \quad \forall n\geq -1$$

    \item [(ii)] The Laplacian associated to the distribution weighted simplicial complex $(K,p)$, is  
    $$L_n = \left(\sum_{i=1}^m \frac{p_i}{1-p_i} \right)  I_{C^n}, \quad \forall n\geq -1.$$
    
\end{enumerate}
\end{corollary}

\begin{proof}
Part (i) is 
immediate from Proposition \ref{prop:IndSimpLap} and Observation \ref{obs:moment}.  
To prove part  (ii), let $L_n^p$ be the Laplacian associated to the independent simplicial distribution weight function. Define $p_\varnothing$ the weight function which to each face assigns it's probability normalized by that of the empty set, as in Observation \ref{obs:norm dist}, and let $L_n^{p_\varnothing}$ be the associated Laplacian. Lemma \ref{lemma:scaling} from Section \ref{sec:Simpl cpx} tells us $L_n^{p_\varnothing}=L_n^{p}$, hence we conclude the desired statement by applying Proposition \ref{prop:IndSimpLap} and Observation  
\end{proof}
\medskip

\section{Proof of Theorem \ref{bigThm}.}\label{sec:seq cpx}
In this section we  provide explicit formulae for  the Laplacian over weighted sequence complexes, and prove  Theorem \ref{bigThm}  by computing the  spectrum and the  associated eigenspaces in the case of the independent vertices model.\\

\subsection{Computation of the Laplacian on Sequences}
We begin by providing the formulae for the coboundary operator and its conjugate on  the sequence complex.
We will consider a weighted sequence complex $(\X,w)$ with weight function $w:\X\to\mathbb{R}_{>0}$, and vertices as $\X_0=V$. Let us introduce a basis $\{e_\sigma\}_{\sigma\in \X} $ in $C^*(\X)$  defined as:
 \begin{equation*} 
 e_{\sigma}(\tau) =\begin{cases}
1,& \text{ if } \sigma= \tau \\
0& \text{ if } \sigma \neq  \tau.
\end{cases}
 \end{equation*} 
The coboundary operator can be written  as
\begin{equation*}
\d_n e_\tau  = 
\sum_{\sigma\in \X_{n+1}}  \kappa({\sigma},{\tau}) e_\sigma.
\end{equation*}

We define an inner product on $C^*(\X)$,  associated to a positive  weight function,  $w$, via its values on the basis elements 
$\{e_{\sigma}\}$: 

 \begin{equation*}
\langle e_{\sigma}, e_{\tau}\rangle \od 
\begin{cases}
w(\sigma) , & \text{ if } \sigma= \tau \\
0& \text{ if } \sigma \neq  \tau.
\end{cases}
 \end{equation*} 

The decomposition $ C^*(\X)=\displaystyle\bigoplus_{n \geq -1} C^n(\X)$ is an orthogonal decomposition with respect to this inner product. 
 
Given this inner product, consider  the adjoint of $\d_n$, the operator $\d_n^*:C^{n+1}(\X)\to C^{n}(\X)$, where $\langle \d_n f,g\rangle =\langle  f,\d_n^*g\rangle $ for  any  $f\in C^{n}(\X) $ and $g\in C^{n+1}(\X)$. 
\begin{equation*}
\d_n^* e_\sigma  = 
\sum_{\tau \in \X_n}  \alpha_\sigma(\tau) e_\tau 
\end{equation*}

for some $\alpha:\X_{n+1}  \times \X_n \to \mathbb{R}$. Following this notation,  for any $e_\eta\in C^n(\X)$:
\begin{align*}
\langle \d_n^* e_\sigma, e_\eta \rangle &= \sum_{\tau \in \X_n}  \alpha_\sigma(\tau)\langle e_\tau, e_\eta \rangle
= \alpha_\sigma(\eta)w(\eta)
\end{align*}

\noindent On the other hand,  by definition of the adjoint, we may rewrite the inner product as:
\begin{align*}
\langle \d_n^* e_\sigma, e_\eta \rangle &= \langle e_\sigma,  \d_n e_\eta \rangle
=\sum_{\omega \in \X_{n+1} }  \kappa({\omega},{\eta})\langle e_\sigma, e_\omega \rangle
= \kappa({\sigma},{\eta}) w(\sigma)
\end{align*}

We conclude the adjoint of the coboundary operator $\d_n$ for all $n\in \mathbb{Z}_{\geq -1}$, can be written  over any basis element,  $e_\sigma \in C^{n+1}(\X)$,  as:
\begin{equation*}
\d_n^* e_\sigma  = 
\sum_{\tau \in \X_n}  \frac{w(\sigma)}{w(\tau)}  \kappa({\sigma},{\tau}) e_\tau
\end{equation*}

\medskip
Our first step in the proof of Theorem \ref{bigThm} will  be to give an explicit formula for the Laplacian on a weighted sequence complex.  We preface this with the following notation. \\

\begin{definition}\label{defs} We introduce a relation on sequences, along with some operations.
\begin{itemize}
\item[(1)]\label{def:swapped} We say $\sigma \bowtie \tau$ if $\sigma, \tau$ have a ``swapped" vertex, that is:
$$\sigma \bowtie \tau \iff \exists \text{ } i \text{ s.t.  } \sigma_j = \tau_j \text{ } \forall j \neq i$$
Notice in particular that we consider $\sigma \bowtie \sigma$.\\

\medskip
\item[(2)]\label{def:glue} We define an  operation to ``glue in" a vertex into a specific spot (thus expanding the length of the sequence):
\begin{align*}
 \X_n \times [0:n] \times \X_0 &\to \X_{n+1}\\
\glue{\sigma}{i}{a} &\mapsto \tau \text{ } \text{ }  \text{s.t.  } \tau_j =
\begin{cases}
 \sigma_j \text{ ,  } \text{ } \text{ } \text{ }   j=0,...,i-1\\
 a \text{ ,  } \text{ } \text{ } \text{ }  \text{ }  j=i\\
 \sigma_{j-1} \text{ ,  } j=i+1,...,n+1\\
\end{cases} 
\end{align*}

\item[(3)]\label{def:remove} As well as an operation  which ``removes" the vertex from a specific spot (thus shortening the length):
\begin{align*}
 \X_n \times [0:n] &\to \X_{n-1}\\
\remove{\sigma}{i} &\mapsto \tau  \text{ } \text{ } \text{s.t.  } \tau_j =
\begin{cases}
 \sigma_j \text{ ,  } \text{ } \text{ } \text{ }   j=0,...,i-1\\
 \sigma_{j+1} \text{ ,  } j=i,...,n-1\\
\end{cases} 
\end{align*}

\item[(4)]\label{def:swap} And finally,  an operation that ``swaps" in a vertex, into a specific spot (thus eliminating the previous vertex in that spot, and preserving the length of the sequence):
\begin{align*}
 \X_n \times [0:n] \times \X_0 &\to \X_n\\
\swap{\sigma}{i}{a} &\mapsto \tau \text{ }  \text{ } \text{ s.t.  } \tau \bowtie \sigma \text{ and } \tau_i=a
\end{align*}

\end{itemize}
\end{definition}

It will be of particular use,  to understand how these functions interact with each other.  We summarize this in the following observation, which is straightforward to verify.\\
\begin{observation}\label{compositions}
Composing the \textit{glue} and \textit{remove} function at the appropriate index yield the \textit{swap} function. The order in which the functions are applied mean different indices must be involved. Given a sequence complex $X$, for any $\sigma \in \X$ and $a \in \X_0$, $i = 0, \dots, \mid \sigma \mid $ :
$$\swap{\sigma}{i}{a}
= \glue{(\remove{\sigma}{i})}{i}{a}
= \remove{(\glue{\sigma}{i}{a})}{i+1}$$

Although the glue and remove functions are not commutative, we may change the order in which they are applied without altering the resulting sequence,  by modifying the involved indices.  Specifically for any $\sigma \in \X$ and $a \in V$, $i = 0, \dots, \dim(\sigma) $,  $j = 0, \dots, \dim(\sigma)+1$ :
\begin{align*}
\remove{(\glue{\sigma}{i}{a})}{j} &= 
\begin{cases}
\glue{(\remove{\sigma}{j})}{i-1}{a} \hspace{12mm} \text{ if } j<i\\
\sigma \hspace{27mm}  \text{ if } j=i\\
\glue{(\remove{\sigma}{j-1})}{i}{a} \hspace{11.5mm} \text{ if } j>i\\
\end{cases} 
\end{align*}
\end{observation}

\medskip
We can rewrite the coboundary operator and it's adjoint in an explicit fashion using these operations. We present these expressions in the next observation.\\
\begin{observation}
Given a weighted sequence complex $(\X,w)$, the coboundary operator $\d_n$,  and it's adjoint $\d_n^*$,  for any $n\in \geq -1$,  can be written over any basis elements $e_\sigma \in C^n(\X)$ and $e_\tau \in C^{n+1}(\X)$ as:
\begin{alignat*}{2}
\d_n e_\sigma &= \sum_{\sigma' \in \X_{n+1}} \kappa(\sigma',\sigma)\cdot e_{\sigma'}
=\sum_{i=0}^{n+1} \sum_{a \in \X_0} (-1)^i \cdot e_{\glue{\sigma}{i}{a}}
\\
\text{ }\\
\d_n^* e_\tau &= \sum_{\tau' \in \X_{n}} \dfrac{w(\tau)}{w(\tau')} \kappa(\tau,\tau')\cdot e_{\tau'}
=\sum_{i=0}^{n+1} \dfrac{w(\tau)}{w(\remove{\tau}{i})} (-1)^i \cdot e_{\remove{\tau}{i}}
\end{alignat*}    
\end{observation}

The observation stems  immediately from the definition of incidence function in Definition \ref{ex:seqcpx}, and of the \textit{gluing} and \textit{removing} operations.  We may apply the observation to compute an explicit definition for the Laplacian operator over a weighted sequence complex. We warn the reader this computation is quite gruesome.\\

\begin{proposition}\label{prop:Laplacian}
Given a weighted sequence complex $(\X,w)$ , over vertices $\X_0=V$,  for all $n \in \mathbb{Z}_{ \geq -1}$,  it's Laplacian of dimension $n$ is given over basis element $e_\sigma \in C^n(\X)$ by:
  \begingroup
\addtolength{\jot}{0.5em}
\begin{align*}
&L_n(e_\sigma) 
=\hspace{5mm} \sum_{i=0}^{n+1} \sum_{a \in V} \dfrac{w(\glue{\sigma}{i}{a})}{w(\sigma)} \cdot e_\sigma 
+ \sum_{i=0}^n \sum_{j=0}^i \sum_{a \in V} (-1)^{i+j} \cdot \left[ \dfrac{w(\sigma)}{w(\remove{\sigma}{j})} - \dfrac{w(\glue{\sigma}{i+1}{a})}{w\left(\glue{(\remove{\sigma}{j})}{i}{a}\right)} \right] \cdot e_{\glue{\left(\remove{\sigma}{j}\right)}{i}{a}}\\
&- \sum_{i=0}^n \sum_{a \in V} \left[ 
\dfrac{w(\glue{\sigma}{i}{a})}{w(\swap{\sigma}{i}{a})}
+\dfrac{w(\glue{\sigma}{i+1}{a})}{w(\swap{\sigma}{i}{a})}
-\dfrac{w(\sigma)}{w(\remove{\sigma}{i})}
\right] \cdot e_{\swap{\sigma}{i}{a}} 
+ \sum_{i=0}^n \sum_{j=i+1}^n \sum_{a \in V} (-1)^{i+j} \cdot \left[ \dfrac{w(\sigma)}{w(\remove{\sigma}{j})} - \dfrac{w(\glue{\sigma}{i}{a})}{w\left(\glue{(\remove{\sigma}{j})}{i}{a}\right)} \right] \cdot e_{\glue{\left(\remove{\sigma}{j}\right)}{i}{a}} 
\end{align*} 
\endgroup
\end{proposition}

\begin{proof}
We compute $L_n^{down}$ and $L_n^{up}$ separately.  Consider a cochain $e_\sigma \in C^n(\X)$,  for $L_n^{down}$ we simply plug in the definitions of $\d_{n-1}$ and $\d_{n-1}^*$, and then split the sum, so that we may substitute $\glue{(\remove{\sigma}{i})}{i}{a} = \swap{\sigma}{i}{a}$ for one of the terms.
  \begingroup
\addtolength{\jot}{0.8em}
\begin{align*}
&L_n^{down} e_\sigma = \d_{n-1} \d_{n-1}^* e_\sigma
= \d_{n-1} \left( \sum_{j=0} ^n (-1)^j \cdot \dfrac{w(\sigma)}{w(\remove{\sigma}{j})} \cdot  e_{\remove{\sigma}{j}} \right)
= \sum_{i=0} ^n \sum_{j=0}^n \sum_{a \in V} (-1)^{i+j} \cdot \dfrac{w(\sigma)}{w(\remove{\sigma}{j})} \cdot  e_{\glue{\left(\remove{\sigma}{j}\right)}{i}{a}}\\
=& \sum_{i=0} ^n \sum_{j=0}^{i-1} \sum_{a \in V} (-1)^{i+j} \cdot \dfrac{w(\sigma)}{w(\remove{\sigma}{j})} \cdot  e_{\glue{\left(\remove{\sigma}{j}\right)}{i}{a}}
 + \sum_{i=0} ^n \sum_{a \in V} (-1)^{2i} \cdot \dfrac{w(\sigma)}{w(\remove{\sigma}{i})} \cdot  e_{\swap{\sigma}{i}{a}}
 + \sum_{i=0} ^n \sum_{j=i+1}^n \sum_{a \in V} (-1)^{i+j} \cdot \dfrac{w(\sigma)}{w(\remove{\sigma}{j})} \cdot  e_{\glue{\left(\remove{\sigma}{j}\right)}{i}{a}}\\
\end{align*}
\endgroup

For $L_n^{up}$ we also begin by plugging in the definitions of $\d_n$ and $\d_n^*$, and then splitting the sum so that we may substitute $\remove{(\glue{\sigma}{i}{a})}{j}=\sigma$ for one of the terms.  Moreover, we also swap $\remove{(\glue{\sigma}{i}{a})}{j}$ for $\glue{(\remove{\sigma}{j})}{i-1}{a}$, when $j<i$; and swap $\remove{(\glue{\sigma}{i}{a})}{j}$ for $\glue{(\remove{\sigma}{j-1})}{i}{a}$, when $j>i$, as per observation \ref{compositions}.
  \begingroup
\addtolength{\jot}{0.8em}
\begin{align*}
&L_n^{up} e_\sigma = \d_n^* \d_n e_\sigma
= \d_n^* \left( \sum_{i=0}^{n+1} \sum_{a \in V} (-1)^i \cdot e_{\glue{\sigma}{i}{a}} \right)
= \sum_{i=0}^{n+1} \sum_{j=0}^{n+1} \sum_{a \in V} (-1)^{i+j} \cdot \dfrac{w(\glue{\sigma}{i}{a})}{w\left(\remove{(\glue{\sigma}{i}{a})}{j}\right)} \cdot e_{\remove{\left(\glue{\sigma}{i}{a}\right)}{j}}\\
=& \sum_{i=0}^{n+1} \sum_{j=0}^{i-1} \sum_{a \in V} (-1)^{i+j} \cdot \dfrac{w(\glue{\sigma}{i}{a})}{w\left(\glue{(\remove{\sigma}{j})}{i-1}{a}\right)} \cdot e_{\glue{\left(\remove{\sigma}{j}\right)}{i-1}{a}}
 + \sum_{i=0}^{n+1} \sum_{a \in V} (-1)^{2i} \cdot \dfrac{w(\glue{\sigma}{i}{a})}{w(\sigma)} \cdot e_{\sigma}\\
&\hspace{60mm}+ \sum_{i=0}^{n+1} \sum_{j=i+1}^{n+1} \sum_{a \in V} (-1)^{i+j} \cdot \dfrac{w(\glue{\sigma}{i}{a})}{w\left(\glue{(\remove{\sigma}{j-1})}{i}{a}\right)} \cdot e_{\glue{\left(\remove{\sigma}{j-1}\right)}{i}{a}}
\end{align*}
\endgroup

We follow by shifting the first sum over $i$ and the third sum over $j$. We then proceeded by eliminating the vacuous terms of the sums, and splitting the first and third sum so as to substitute in $\glue{(\remove{\sigma}{i})}{i}{a} = \swap{\sigma}{i}{a}$.
  \begingroup
\addtolength{\jot}{0.8em}
\begin{align*}
L_n^{up} e_\sigma &= \sum_{i=-1}^{n} \sum_{j=0}^{i} \sum_{a \in V} (-1)^{i+j+1} \cdot \dfrac{w(\glue{\sigma}{i+1}{a})}{w\left(\glue{(\remove{\sigma}{j})}{i}{a}\right)} \cdot e_{\glue{\left(\remove{\sigma}{j}\right)}{i}{a}}
 +  
\sum_{i=0}^{n+1} \sum_{a \in V}  \dfrac{w(\glue{\sigma}{i}{a})}{w(\sigma)} \cdot e_{\sigma}\\
&\hspace{2mm} + \sum_{i=0}^{n+1} \sum_{j=i}^{n} \sum_{a \in V} (-1)^{i+j+1} \cdot \dfrac{w(\glue{\sigma}{i}{a})}{w\left(\glue{(\remove{\sigma}{j})}{i}{a}\right)} \cdot e_{\glue{\left(\remove{\sigma}{j}\right)}{i}{a}}\\
&= \sum_{i=0}^{n} \sum_{j=0}^{i-1} \sum_{a \in V} (-1)^{i+j+1} \cdot \dfrac{w(\glue{\sigma}{i+1}{a})}{w\left(\glue{(\remove{\sigma}{j})}{i}{a}\right)} \cdot e_{\glue{\left(\remove{\sigma}{j}\right)}{i}{a}}
 + \sum_{i=0}^{n} \sum_{a \in V} (-1)^{2i+1} \cdot
 \dfrac{w(\glue{\sigma}{i+1}{a})}{w(\swap{\sigma}{i}{a})} 
 \cdot e_{\swap{\sigma}{i}{a}}\\
&\hspace{2mm} + 
\sum_{i=0}^{n+1} \sum_{a \in V}  \dfrac{w(\glue{\sigma}{i}{a})}{w(\sigma)} \cdot e_{\sigma} + \sum_{i=0}^{n} \sum_{a \in V} (-1)^{2i+1} \cdot
  \dfrac{w(\glue{\sigma}{i}{a})}{w(\swap{\sigma}{i}{a})} 
 \cdot e_{\swap{\sigma}{i}{a}}\\
&\hspace{2mm} + \sum_{i=0}^{n} \sum_{j=i+1}^{n} \sum_{a \in V} (-1)^{i+j+1} \cdot \dfrac{w(\glue{\sigma}{i}{a})}{w\left(\glue{(\remove{\sigma}{j})}{i}{a}\right)} \cdot e_{\glue{\left(\remove{\sigma}{j}\right)}{i}{a}}\\
\end{align*}
\endgroup

Adding our computations for $L_n^{down}$ and $L_n^{up}$,  we conclude the theorem statement.
\end{proof}

\begin{corollary} \label{cor:IndLaplacian}
Given a weighted full sequence complex $(\X,w)$, over vertices $\X_0=V$, following the independent vertices model (Definition \ref{def:ind seq}), the Laplacian of dimension $n$, for all $n\in \mathbb{Z}_{\geq -1}$,  can be simply described over basis element $e_\sigma \in C^n(\X)$ as:
\begin{equation*}
L_n(e_\sigma) =\left[ (n+2) -\sum_{i=0}^n w(\sigma_i) \right] \cdot e_\sigma
- \sum_{\substack{\tau :\tau \bowtie \sigma,\\ \tau_j \neq \sigma_j}} w(\sigma_j) \cdot e_{\tau}
\end{equation*} 
\end{corollary}

\begin{proof}

Let us simplify the terms that appear in proposition \ref{prop:Laplacian}. Fix $\sigma \in \X_n$ for some $n \in \ZZ_{\geq -1}$,  and $i \in [0:n]$,  $a \in V$, by applying the independence model hypothesis we can write out:
 \begingroup
\addtolength{\jot}{1em}
\begin{alignat*}{2}
 \dfrac{w(\glue{\sigma}{i}{a})}{w(\sigma)}  = \dfrac{\Big(\prod_{j=0}^{i-1} w(\sigma_j) \Big) \cdot w(a) \cdot \Big(\prod_{j=i}^n w(\sigma_j)\Big)}{\prod_{j=0}^n w(\sigma_j)}= w(a) 
\end{alignat*}
\endgroup

Similarly, we may write out and cancel terms to get:
\medskip
 \begingroup
\addtolength{\jot}{1em}
\begin{alignat*}{2}
w(\sigma_i)&= \dfrac{w(\sigma)}{w(\remove{\sigma}{i})}
= \dfrac{w(\glue{\sigma}{i}{a})}{w(\swap{\sigma}{i}{a})}
= \dfrac{w(\glue{\sigma}{i+1}{a})}{w(\swap{\sigma}{i}{a})} \\
w(\sigma_j) &= \dfrac{w(\sigma)}{w(\remove{\sigma}{j})} 
= \dfrac{w(\glue{\sigma}{i+1}{a})}{w\left(\glue{(\remove{\sigma}{j})}{i}{a}\right)}
=  \dfrac{w(\glue{\sigma}{i}{a})}{w\left(\glue{(\remove{\sigma}{j})}{i}{a}\right)}
\end{alignat*}
\endgroup

\medskip
Substituting these equalities into the expression from Proposition \ref{prop:Laplacian},  simplifies it to:
\begin{align}
L_n(e_\sigma) =& \sum_{i=0}^{n+1} \sum_{a \in V} w(a) \cdot e_\sigma
- \sum_{i=0}^n \sum_{a \in V} \big[ w(\sigma_i) + 0 \big] \cdot e_{\swap{\sigma}{i}{a}} 
+ \sum_{i=0}^n \sum_{j=0}^i \sum_{a \in V}  0  \cdot e_{\glue{\left(\remove{\sigma}{j}\right)}{i}{a}} 
+ \sum_{i=0}^n \sum_{j=i+1}^n \sum_{a \in V} 0  \cdot e_{\glue{\left(\remove{\sigma}{j}\right)}{i}{a}} \nonumber\\
=&(n+2) \cdot e_\sigma
- \sum_{i=0}^n \sum_{a \in V} w(\sigma_i) \cdot e_{\swap{\sigma}{i}{a}}
=\left[ (n+2) -\sum_{i=0}^n w(\sigma_i) \right] \cdot e_\sigma
- \sum_{\substack{\tau :\tau \bowtie \sigma,\\ \tau_j \neq \sigma_j}} w(\sigma_j) \cdot e_{\tau} \nonumber
\end{align}
Where  the second equality follows from $\sum_{v\in V} w(v)=1$. This concludes the corollary statement.
\end{proof}

\medskip
\subsection{Eigenspace for Laplacian on  Sequences under  Independent Vertices Model }
For the study of the independent vertices model,  it is useful to introduce a tensor product on the space of cochains, the definition follows:\\
\begin{definition}
   For cochains $u \in C^k(\X)$ and $v \in C^l(\X)$, the \emph{tensor product} $u \otimes v \in C^{k+l+1}(\X)$ is the cochain whose value on a sequence $\sigma \in \X_{k+l+1}$ is given by the formula:
$$(u \otimes v)(\sigma)=u(\sigma \vert_{[0 : k]}) \cdot v(\sigma \vert_{[k+1 : k+l+1]})$$
Here $\sigma |_{[i:j]}$ refers to the restriction of the sequence $\sigma$ to the vertices in it's $i$ through $j$ spots (preserving the ordering),  yielding a subsequence of length $j-i+1$. 
\end{definition}
 Note that this tensor product is reminiscent of, but different from  the cup product\footnote{
The tensor product $\otimes : C^k(\X) \times C^l(\X) \to C^{k+l+1}(\X)$ and  cup product $\smile : C^k(\X) \times C^l(\X) \to C^{k+l}(\X)$ differ in the dimension of the resulting cochain. Recall that $u \smile v \in C^{{k+l}}(\X)$ takes value on  a sequence $\tau \in \X_{{k+l}}$ as follows:
$$(u \smile v)(\tau)=u(\tau \vert_{[0 : k]}) \cdot v(\tau \vert_{[{k} : k+l]})$$
}.
Using the tensor product on the space of cochains, we introduce a new basis for $C^n(\X)$.

\begin{definition}\label{def:eig basis}
Fix $a \in \X_0$ and $w:\X \to (0,1]$ a weight function. We define the associated  map $$f_0= f_0^{a,w} : \X_0 \to C^0(\X)$$ associating to all $x \in \X_0$,  a $0$-cochain defined over each $y \in \X_0$ as:
\begin{alignat*}{2}
f_0(x) : \X_0 &\to \R\\
y &\mapsto \big( f_0(x) \big) (y)&&= \1_a(x) + \Big(1- \1_a(x)\Big) \cdot \Big(w(x) \1_a(y)-w(a) \1_x(y)\Big)\\
& &&=
\begin{cases}
1 \hspace{11.5mm} \text{ if } x=a\\
w(x) \hspace{6mm} \text{ if } y=a \neq x\\
-w(a) \hspace{3.5mm}  \text{ if } y=x \neq a\\
0  \hspace{12mm} \text{ if not}
\end{cases}
\end{alignat*}

\medskip
We further define the map $$f= f^{a,w} : \X \to C^*(\X)$$
which, for all $n\in \ZZ_{\geq -1}$, maps any $n$-dimensional sequence, $\eta \in \X_n$,  to an $n$-cochain:
$$f(\eta) = f_0(\eta_1) \otimes \dots \otimes f_0(\eta_n) \text{      } \in C^n(\X)$$
Here we denote by $\eta_i \in \X_0$ the $i$'th vertex of the sequence $\eta$.
\end{definition}

We give an example of a particular cochain $f(\eta)\in C^6(\X)$ applied to two different sequences $\sigma,\tau \in \X_6$ in Figure \ref{ex:eig}.\\

\begin{figure}[ht]
\centering
\begin{subfigure}[b]{0.45\textwidth}
\begin{tikzpicture}[
 add paren/.style={
      left delimiter={(},
      right delimiter={)}, 
    }]

\matrix(m) [column sep=7mm,   box/.style={
    shape=rectangle,
    text width=0.6cm,
    minimum height=0.7cm,
    text centered,
    draw=black,
    outer sep=0.1cm,
    font=\large,
    font=\bf
},  
boxb/.style={
    shape=rectangle,
    text width=0.6cm,
    minimum height=0.7cm,
    text centered,
    draw=black,
    fill=lblue,
    outer sep=0.1cm,
    font=\large,
    font=\bf
}, 
boxh/.style={
    shape=rectangle,
    text width=0.6cm,
    minimum height=0.7cm,
    text centered,
    draw=black,
   pattern={mylines[angle=45,line width=2.5pt]}, 
   pattern color=lgray,
    outer sep=0.1cm,
    font=\large,
    font=\bf
}, 
float/.style={
    shape=rectangle,
    text width=1.2cm,
    minimum height=0.7cm,
    text centered,
    outer sep=0.1cm,
     font=\large
},
empty/.style={
    shape=rectangle,
    text width=0cm,
    minimum height=0.7cm,
    text centered,
    outer sep=0cm,
     font=\large
},] 
{
&
 \node[float] (z1) {$\eta$}; &
 \node[float] (z2) {$\sigma$};&
 \node[float] (z3) { };
  \\
  &
  \node[boxh] (a1) {$\eta_0$};&
   \node[box] (a2) {\textcolor{myGreen}{$\eta_0$}}; &
    \node[float] (a3) {$-w(a)$};\\
 &
  \node[boxh] (b1) {$\eta_1$};&
   \node[box] (b2) {\textcolor{BlueViolet}{$a$}}; &
    \node[float] (b3) {$w(\eta_1)$};\\
  &
\node[boxh] (c1) {$\eta_2$};&
   \node[box] (c2) {\textcolor{myGreen}{$\eta_2$}}; &
    \node[float] (c3) {$-w(a)$};\\
   \node[empty] (d0) { };&   
\node[boxh] (d1) {$\eta_3$};&
   \node[box] (d2) {\textcolor{myGreen}{$\eta_3$}}; &
    \node[float] (d3) {$-w(a)$};
     \node[empty] (d4) { };&   \\
          &
\node[boxh] (e1) {$\eta_4$};&
   \node[box] (e2) {\textcolor{myGreen}{$\eta_4$}}; &
    \node[float] (e3) {$-w(a)$};\\
          &
\node[boxb] (f1) {$a$};&
   \node[box] (f2) {\textcolor{myBlue}{$a$}}; &
    \node[float] (f3) {$1$};\\
              &
\node[boxb] (g1) {$a$};&
   \node[box] (g2) {\textcolor{myBlue}{$\sigma_6$}}; &
    \node[float] (g3) {$1$};\\
    &&&
     \node[float] (h3) {$\shortparallel$};\\
      &&&
     \node[empty] (i3) {};\\
};
\node[fit=(a1) (g1), add paren] (matrix){};
\node (label) at ($(d0)!0.2!(d1)$) {{\Large $f$}};
\node (label) at ($(d1)!0.6!(d2)$) {:};
\draw[thick,shorten >=0.1cm,shorten <=0.2cm, |->] (d2)--(d3);
\node (label) at ($(d3)!0.6!(d4)$) {\hspace{35mm}\large{,}};
\node (label) at ($(z3)!0.3!(a3)$) {$\rotatebox{270}{\Bigg(}$};
\node (label) at ($(a3)!0.5!(b3)$) {\Large{$\cdot$}};
\node (label) at ($(b3)!0.5!(c3)$) {\Large{$\cdot$}};
\node (label) at ($(c3)!0.5!(d3)$) {\Large{$\cdot$}};
\node (label) at ($(d3)!0.5!(e3)$) {\Large{$\cdot$}};
\node (label) at ($(e3)!0.5!(f3)$) {\Large{$\cdot$}};
\node (label) at ($(f3)!0.5!(g3)$) {\Large{$\cdot$}};
\node (label) at ($(g3)!0.5!(h3)$) {$\rotatebox{270}{\Bigg)}$};
\node (label) at ($(h3)!0.7!(i3)$) {{\large $w(a)^4 \cdot w(\eta_1)$}};
\end{tikzpicture}
\end{subfigure}
   \hfill
  \begin{subfigure}[b]{0.45\textwidth}
\begin{tikzpicture}[
  add paren/.style={
      left delimiter={(},
      right delimiter={)}, 
    }]

\matrix(m) [column sep=7mm,   box/.style={
    shape=rectangle,
    text width=0.6cm,
    minimum height=0.7cm,
    text centered,
    draw=black,
    outer sep=0.1cm,
    font=\large,
    font=\bf
},  
boxb/.style={
    shape=rectangle,
    text width=0.6cm,
    minimum height=0.7cm,
    text centered,
    draw=black,
    fill=lblue,
    outer sep=0.1cm,
    font=\large,
    font=\bf
}, 
boxh/.style={
    shape=rectangle,
    text width=0.6cm,
    minimum height=0.7cm,
    text centered,
    draw=black,
   pattern={mylines[angle=45,line width=2.5pt]}, 
   pattern color=lgray,
    outer sep=0.1cm,
    font=\large,
    font=\bf
}, 
float/.style={
    shape=rectangle,
    text width=1.2cm,
    minimum height=0.7cm,
    text centered,
    outer sep=0.1cm,
     font=\large
},
empty/.style={
    shape=rectangle,
    text width=0cm,
    minimum height=0.7cm,
    text centered,
    outer sep=0cm,
     font=\large
},] 
{
&
 \node[float] (z1) {$\eta$}; &
 \node[float] (z2) {$\tau$};&
 \node[float] (z3) { };
  \\
  &
  \node[boxh] (a1) {$\eta_0$};&
   \node[box] (a2) {\textcolor{myGreen}{$\eta_0$}}; &
    \node[float] (a3) {$-w(a)$};\\
 &
  \node[boxh] (b1) {$\eta_1$};&
   \node[box] (b2) {\textcolor{myGreen}{$\eta_1$}}; &
    \node[float] (b3) {$-w(a)$};\\
  &
\node[boxh] (c1) {$\eta_2$};&
   \node[box] (c2) {\textcolor{Maroon}{$\tau_2$}}; &
    \node[float] (c3) {$0$};\\
   \node[empty] (d0) { };&   
\node[boxh] (d1) {$\eta_3$};&
   \node[box] (d2) {\textcolor{BlueViolet}{$a$}}; &
    \node[float] (d3) {$w(\eta_2)$};\\
          &
\node[boxh] (e1) {$\eta_4$};&
   \node[box] (e2) {\textcolor{myGreen}{$\eta_4$}}; &
    \node[float] (e3) {$-w(a)$};\\
          &
\node[boxb] (f1) {$a$};&
   \node[box] (f2) {\textcolor{myBlue}{$\tau_5$}}; &
    \node[float] (f3) {$1$};\\
              &
\node[boxb] (g1) {$a$};&
   \node[box] (g2) {\textcolor{myBlue}{$\tau_6$}}; &
    \node[float] (g3) {$1$};\\
    &&&
     \node[float] (h3) {$\shortparallel$};\\
      &&&
     \node[empty] (i3) {};\\
};
\node[fit=(a1) (g1), add paren] (matrix){};
\node (label) at ($(d0)!0.2!(d1)$) {{\Large $f$}};
\node (label) at ($(d1)!0.6!(d2)$) {:};
\draw[thick,shorten >=0.1cm,shorten <=0.2cm, |->] (d2)--(d3);
\node (label) at ($(z3)!0.3!(a3)$) {$\rotatebox{270}{\Bigg(}$};
\node (label) at ($(a3)!0.5!(b3)$) {\Large{$\cdot$}};
\node (label) at ($(b3)!0.5!(c3)$) {\Large{$\cdot$}};
\node (label) at ($(c3)!0.5!(d3)$) {\Large{$\cdot$}};
\node (label) at ($(d3)!0.5!(e3)$) {\Large{$\cdot$}};
\node (label) at ($(e3)!0.5!(f3)$) {\Large{$\cdot$}};
\node (label) at ($(f3)!0.5!(g3)$) {\Large{$\cdot$}};
\node (label) at ($(g3)!0.5!(h3)$) {$\rotatebox{270}{\Bigg)}$};
\node (label) at ($(h3)!0.7!(i3)$) {{\large $0$}};
\end{tikzpicture}
\end{subfigure}

\caption{Example cochain $f(\eta) \in C^6(\X)$,  
$\eta=(\eta_0,\eta_1,\eta_2,\eta_3,\eta_4,a,a)\in \X_6$ 
applied to two different sequences  $\sigma=(\eta_0,a,\eta_2,\eta_3,\eta_4,a,\sigma_6)\in \X_6$,  and $\tau=(\eta_0,\eta_1,\tau_2,a,\eta_4,\tau_5,\tau_6)\in \X_6$, with $\tau_2 \neq \eta_2$ (in red). We observe $\left(f(\eta)\right)(\sigma) \neq 0$ whereas $\left(f(\eta)\right)(\tau) = 0$. }\label{ex:eig}
\end{figure}

Of interest to us, is that the cochains  $\{ f(\eta) \}_{\eta \in \X_n }$ form a basis for $C^n(\X)$, the proof is quite simple, we give it in the following lemma.\\

\begin{lemma} Given a sequence complex $\X$, and any $n \in \ZZ_{\geq 0}$,
the cochains $\{ f(\eta) \}_{\eta \in \X_n }$ form a basis for $C^n(\X)$.
\end{lemma}
\begin{proof}
We will show $\Span \big\{  f(\eta) \mid \eta \in  \X_n  \big\} \ni e_\sigma$, for all $ \sigma \in \X_n$, $n \in \ZZ_{\geq 0}$.  We do the proof by induction over $n$.\\

Consider n=0, then $\X_0= V = \{a,b,c,...\}$. By definition we have: 
\begin{align*}
f({(a)})&=
\big(f_0(a)\big)(a)\cdot e_{(a)} 
+ \big(f_0(a)\big)(b) \cdot e_{(b)}
+ \big(f_0(a)\big)(c) \cdot e_{(c)}
+\dots 
=e_{(a)} + e_{(b)} + e_{(c)} + \dots\\
f({(b)})&=
\big(f_0(b)\big)(a)\cdot e_{(a)} 
+ \big(f_0(b)\big)(b) \cdot e_{(b)}
+ \big(f_0(b)\big)(c) \cdot e_{(c)}
+\dots 
=w(b) \cdot e_{(a)}  -w(a) \cdot e_{(b)} \\
f({(c)})&=
\big(f_0(c)\big)(a)\cdot e_{(a)} 
+ \big(f_0(c)\big)(b) \cdot e_{(b)}
+ \big(f_0(c)\big)(c) \cdot e_{(c)}
+\dots 
=w(c) \cdot e_{(a)}  -w(a) \cdot e_{(c)} \\
\vdots&
\end{align*}

\medskip
As we have $w(v)\neq0$ for all $v \in V$ ,  we may take the following linear combination of  cochains:
\begin{align*}
f({(a)}) +\sum_{x \neq a} w(a)^{-1} \cdot f({(x)}) 
= \sum_{x \in A} e_{(x)} + \sum_{x \neq a} \Big( \dfrac{w(x)}{w(a)} \cdot e_{(a)} -e_{(x)} \Big)
= \Bigg( 1+ \frac{1-w(a)}{w(a)} \Bigg) \cdot e_{(a)}
=w(a)^{-1} \cdot e_{(a)}
\end{align*}
Notice that in the second equality, we made use of $\sum_{x \in V}w(x)=1$. Proving $e_{(a)} \in \Span \big\{  f(\eta) \mid \eta \in  \X_n  \big\}$.\\

On the other hand, by definition, we may write $f({(y)})$ for any $y \in V$, $y \neq a$ as:
$$ f({(y)}) = w(y) \cdot e_{(a)} - w(a) \cdot e_{(y)} $$
Solving for $e_{(y)}$ and plugging in our identity for $e_{(a)}$, we get 
\begin{align*}
 e_{(y)}
=-w(a)^{-1} \cdot \big(f({(y)})- w(y) \cdot e_{(a)} \big) 
=-w(a)^{-1} \cdot f({(y)})+  w(y) \cdot \Bigg( f({(a)}) +\sum_{x \neq a} w(a)^{-1} \cdot f({(x)}) \Bigg)
\end{align*}

Proving for all $y \in X_0$, $e_{(y)} \in  \Span \big\{  f(\eta) \mid \eta \in  \X_0 \big\}$,  concluding the base induction case.\\

Consider $\eta \in \X_n$ and denote $\eta_0 \in \X_0$ it's first vertex and $\eta \mid_{[1,n]} \in \X_{n-1}$ it's restriction to it's other $n$ vertices.  By definition,  $f(\eta)= f({\eta_0}) \otimes f({\eta \mid_{[1,n]}})$, allowing us to immediately apply our induction hypothesis:
\begin{align*}
\Span\{  f(\eta) \mid \eta \in  \X_n+1  \}
&= \Span\{  f({(x)}) \mid x \in  \X_0   \} \otimes \Span\{  f(\tau) \mid \tau \in  \X_n  \}\\
&= \Span\{  e_{(x)} \mid x \in  \X_0  \} \otimes \Span\{   e_\tau \mid \tau \in  \X_n  \}
= \Span\{  e_\eta \mid \eta \in  \X_{n+1}   \}\\
\end{align*}
proving the lemma.
\end{proof}

Having proven the cochains $\{ f(\eta) : \eta \in \X_n \}$ form a basis for $C^n(\X)$, we can finally detail the proof for Theorem \ref{bigThm}, presented in abridged form in the previous section.  We give here the full version of the theorem, which exposes the eigenspaces associated to each eigenvalue.\\

\begin{reptheorem}{bigThm}
The Laplacian of a weighted full sequence complex $(\X,w)$, over vertices $\X_0=V$, with weights following the independent vertices model (Definition \ref{def:ind seq}):
\begin{align*}
L_n : \hspace{1mm} &C^n(\X) \to C^n(X) \hspace{5mm} \text{for } n\in \mathbb{Z}_{\geq 0}\\
\text{admits  eigenvalues} \hspace{50mm}&\\
\lambda= \hspace{1mm}  &1,...,n+2\\
\text{with multiplicity:} \hspace{53mm}&\\
\text{mult}(\lambda)= \hspace{1mm}  &\binom{n+1}{\lambda-1} \cdot (|V|-1)^{\lambda-1} \hspace{50mm}
\end{align*}
By fixing one of the vertices,  $ a\in V$ , we can succinctly describe the eigenspaces,  associated over the usual basis $\{ e_\sigma \}$ ,  to each of our  eigenvalues:
\begin{align*}
E(\lambda,L_n)&=\Span \big\{  f(\eta) \mid \eta \in  \X_n : \# \{a \in \eta \} =n+2 - \lambda \big\} \hspace{5mm} \text{,  }  \lambda=1,...,n+2
\end{align*}
Where $f(\eta)$ is as in Definition \ref{def:eig basis}, and $\#\{a\in \eta\}$ denotes the number of vertex ``$a$" in the sequence $\eta$.
\end{reptheorem}

\begin{proof}
Let us show the cochains $f(\eta)$ are eigenvectors for the Laplacian associated to a independent vertices model sequence complex under the usual representation.  From corollary \ref{cor:IndLaplacian}, we know the Laplacian $L_n$, for all $n \in \ZZ \geq -1$, has representation, over the usual basis $\{ e_\sigma \}$, given by:

\begin{align*}
(L_n)_{\tau \sigma} =
\begin{cases}
(n+2) -\sum_{i=0}^n w(\sigma_i) \hspace{3mm} &\text{if } \tau=\sigma\\
w(\sigma_j)  &\text{if } \tau \bowtie \sigma \text{ ,  } \tau_j \neq \sigma_j\\ 
0 &\text{if } \tau \not\bowtie \sigma
\end{cases}
\end{align*}

\noindent for any $\sigma, \tau \in \X_n$. We recall the use here of the ``swapped" sequences notation $\tau \bowtie \sigma$ from Definition\ref{defs}.\\

Hence to show  $v \in C^n(\X)$ is an eigenvector of $L_n$ associated to $\lambda$,  we need to prove:
\begin{align}\label{eigvec}
\lambda \cdot v(\tau) =\left[ (n+2) -\sum_{i=0}^n w(\tau_i) \right] \cdot v(\tau)
- \sum_{\substack{\sigma :\sigma \bowtie \tau,\\ \tau_j \neq \sigma_j}} w(\sigma_j) \cdot v({\sigma}) \hspace{3mm} \text{ ,  } \forall \tau \in \X_n
\end{align}

We will separate the proof into two parts. First we will deal with the cochains of the form $f\big({(a, \dots , a)}\big) \in C^n(\X)$, for some $n \in \ZZ_{\geq0}$. We illustrate such a cochain in Figure \ref{ex:eig all a}.\\

\begin{figure}[ht]
\centering
\begin{tikzpicture}[
 add paren/.style={
      left delimiter={(},
      right delimiter={)}, 
    }]

\matrix(m) [column sep=0.5mm,   box/.style={
    shape=rectangle,
    text width=0.3cm,
    minimum height=0.6cm,
    text centered,
    draw=black,
    font=\large,
    font=\bf
},  
boxb/.style={
    shape=rectangle,
    text width=0.3cm,
    minimum height=0.6cm,
    text centered,
    draw=black,
    fill=lblue,
    font=\large,
    font=\bf
}, 
boxh/.style={
    shape=rectangle,
    text width=0.3cm,
    minimum height=0.6cm,
    text centered,
    draw=black,
   pattern={mylines[angle=45,line width=2.5pt]}, 
   pattern color=lgray,
    font=\large,
    font=\bf
}, 
float/.style={
    shape=rectangle,
    text width=1.2cm,
    minimum height=0.6cm,
    text centered,
    outer sep=0.1cm,
},
empty/.style={
    shape=rectangle,
    text width=0cm,
    minimum height=0.6cm,
},] 
{
&
 \node[float] (z1) {$\eta$}; &
 \node[float] (z2) {$\sigma$};&
 \node[float] (z3) { };
  \\
  &
  \node[boxb] (a1) {$a$};&
   \node[box] (a2) {\textcolor{myBlue}{$\sigma_0$}}; &
    \node[float] (a3) {$1$};\\
 &
  \node[boxb] (b1) {$a$};&
   \node[box] (b2) {\textcolor{myBlue}{$\sigma_1$}}; &
    \node[float] (b3) {$1$};\\
   \node[empty] (c0) { };&  
\node[boxb] (c1) {$a$};&
   \node[box] (c2) {\textcolor{myBlue}{$\sigma_2$}}; &
    \node[float] (c3) {$1$};&
       \node[empty] (c4) { }; \\
          &
\node[boxb] (d1) {$a$};&
   \node[box] (d2) {\textcolor{myBlue}{$\sigma_3$}}; &
    \node[float] (d3) {$1$};\\
              &
\node[boxb] (e1) {$a$};&
   \node[box] (e2) {\textcolor{myBlue}{$\sigma_4$}}; &
    \node[float] (e3) {$1$};\\
    &&&
     \node[float] (f3) {$\shortparallel$};\\
      &&&
     \node[empty] (g3) {};\\
};
\node[fit=(a1) (e1), add paren] (matrix){};
\node (label) at ($(c0)!0!(c1)$) {{\Large $f$}};
\node (label) at ($(c1)!0.6!(c2)$) {:};
\draw[thick,shorten >=0cm,shorten <=0.1cm, |->] (c2)--(c3);

\node (label) at ($(z3)!0.3!(a3)$) {$\rotatebox{270}{\Bigg(}$};
\node (label) at ($(a3)!0.5!(b3)$) {\Large{$\cdot$}};
\node (label) at ($(b3)!0.6!(c3)$) {\Large{$\cdot$}};
\node (label) at ($(c3)!0.5!(d3)$) {\Large{$\cdot$}};
\node (label) at ($(d3)!0.5!(e3)$) {\Large{$\cdot$}};
\node (label) at ($(e3)!0.5!(f3)$) {$\rotatebox{270}{\Bigg)}$};
\node (label) at ($(f3)!1!(g3)$) {{$1$}};
\end{tikzpicture}

\caption{Example of a cochain associated to a sequence of all  $a$'s. Figured here $\eta=(a,a,a,a,a)\in \X_4$ is applied to a sequence $\sigma=(\sigma_0,\sigma_1,\sigma_2,\sigma_3,\sigma_4) \in \X_4$ and yields $\big( f(\eta)\big) (\sigma)=1$. }\label{ex:eig all a}
\end{figure}

\newpage
After showing the cochain $f({(a, \dots , a)}) \in C^n(\X)$ is an eigenvector, we will show all other cochains $f(\eta) \in C^n(\X)$, $\eta \neq (a, \dots , a)$, are eigenvectors for the Laplacian $L_n$ as well. In this general step, we distinguish the action of the cochain $f(\eta)$ on $\tau\in\X_n$ such that $\big( f(\eta) \big) (\tau)=0$ and on $\sigma\in\X_n$ such that $\big( f(\eta) \big) (\sigma) \neq 0$. In both cases, we explore what the action of $f(\eta)$ would be on any  $\tau' \bowtie \tau$ and $\sigma' \bowtie \sigma$. We illustrate some examples for the first case in Figure \ref{ex:eig zero} and for the second case in Figure \ref{ex:eig nonzero}.\\

\begin{figure}[ht]
\centering
\begin{subfigure}[b]{0.3\textwidth}
\begin{tikzpicture}[
 add paren/.style={
      left delimiter={(},
      right delimiter={)}, 
    }]

\matrix(m) [column sep=0.5mm,   box/.style={
    shape=rectangle,
    text width=0.3cm,
    minimum height=0.6cm,
    text centered,
    draw=black,
    font=\large,
    font=\bf
},  
boxb/.style={
    shape=rectangle,
    text width=0.3cm,
    minimum height=0.6cm,
    text centered,
    draw=black,
    fill=lblue,
    font=\large,
    font=\bf
}, 
boxh/.style={
    shape=rectangle,
    text width=0.3cm,
    minimum height=0.6cm,
    text centered,
    draw=black,
   pattern={mylines[angle=45,line width=2.5pt]}, 
   pattern color=lgray,
    font=\large,
    font=\bf
}, 
float/.style={
    shape=rectangle,
    text width=1.2cm,
    minimum height=0.6cm,
    text centered,
    outer sep=0.1cm,
},
empty/.style={
    shape=rectangle,
    text width=0cm,
    minimum height=0.6cm,
},] 
{
&
 \node[float] (z1) {$\eta$}; &
 \node[float] (z2) {$\tau$};&
 \node[float] (z3) { };
  \\
  &
  \node[boxh] (a1) {$\eta_0$};&
   \node[box] (a2) {\textcolor{myGreen}{$\eta_0$}}; &
    \node[float] (a3) {$-w(a)$};\\
 &
  \node[boxh] (b1) {$\eta_1$};&
   \node[box] (b2) {\textcolor{BlueViolet}{$a$}}; &
    \node[float] (b3) {$w(\eta_1)$};\\
   \node[empty] (c0) { };&  
\node[boxh] (c1) {$\eta_2$};&
   \node[box] (c2) {\textcolor{Maroon}{$\tau_2$}}; &
    \node[float] (c3) {$0$};&
       \node[empty] (c4) { }; \\
          &
\node[boxb] (d1) {$a$};&
   \node[box] (d2) {\textcolor{myBlue}{$\tau_3$}}; &
    \node[float] (d3) {$1$};\\
              &
\node[boxb] (e1) {$a$};&
   \node[box] (e2) {\textcolor{myBlue}{$\tau_4$}}; &
    \node[float] (e3) {$1$};\\
    &&&
     \node[float] (f3) {$\shortparallel$};\\
      &&&
     \node[empty] (g3) {};\\
};
\node[fit=(a1) (e1), add paren] (matrix){};
\node (label) at ($(c0)!0!(c1)$) {{\Large $f$}};
\node (label) at ($(c1)!0.6!(c2)$) {:};
\draw[thick,shorten >=0cm,shorten <=0.1cm, |->] (c2)--(c3);

\node (label) at ($(z3)!0.3!(a3)$) {$\rotatebox{270}{\Bigg(}$};
\node (label) at ($(a3)!0.5!(b3)$) {\Large{$\cdot$}};
\node (label) at ($(b3)!0.6!(c3)$) {\Large{$\cdot$}};
\node (label) at ($(c3)!0.5!(d3)$) {\Large{$\cdot$}};
\node (label) at ($(d3)!0.5!(e3)$) {\Large{$\cdot$}};
\node (label) at ($(e3)!0.5!(f3)$) {$\rotatebox{270}{\Bigg)}$};
\node (label) at ($(f3)!1!(g3)$) {{$0$}};
\end{tikzpicture}
\end{subfigure}

\begin{subfigure}[b]{0.3\textwidth}
\begin{tikzpicture}[
 add paren/.style={
      left delimiter={(},
      right delimiter={)}, 
    }]

\matrix(m) [column sep=0.5mm,   box/.style={
    shape=rectangle,
    text width=0.3cm,
    minimum height=0.6cm,
    text centered,
    draw=black,
    font=\large,
    font=\bf
},  
boxb/.style={
    shape=rectangle,
    text width=0.3cm,
    minimum height=0.6cm,
    text centered,
    draw=black,
    fill=lblue,
    font=\large,
    font=\bf
}, 
boxy/.style={
    shape=rectangle,
    text width=0.3cm,
    minimum height=0.6cm,
    text centered,
    draw=black,
    fill=lyellow,
    font=\large,
    font=\bf
}, 
boxh/.style={
    shape=rectangle,
    text width=0.3cm,
    minimum height=0.6cm,
    text centered,
    draw=black,
   pattern={mylines[angle=45,line width=2.5pt]}, 
   pattern color=lgray,
    font=\large,
    font=\bf
}, 
float/.style={
    shape=rectangle,
    text width=1.2cm,
    minimum height=0.6cm,
    text centered,
    outer sep=0.1cm,
},
empty/.style={
    shape=rectangle,
    text width=0cm,
    minimum height=0.6cm,
},] 
{
&
 \node[float] (z1) {$\eta$}; &
 \node[float] (z2) {$\tau'$};&
 \node[float] (z3) { };
  \\
  &
  \node[boxh] (a1) {$\eta_0$};&
   \node[boxy] (a2) {\textcolor{BlueViolet}{$a$}}; &
    \node[float] (a3) {$w(\eta_0)$};\\
 &
  \node[boxh] (b1) {$\eta_1$};&
   \node[box] (b2) {\textcolor{BlueViolet}{$a$}}; &
    \node[float] (b3) {$w(\eta_1)$};\\
   \node[empty] (c0) { };&  
\node[boxh] (c1) {$\eta_2$};&
   \node[box] (c2) {\textcolor{Maroon}{$\tau_2$}}; &
    \node[float] (c3) {$0$};&
       \node[empty] (c4) { }; \\
          &
\node[boxb] (d1) {$a$};&
   \node[box] (d2) {\textcolor{myBlue}{$\tau_3$}}; &
    \node[float] (d3) {$1$};\\
              &
\node[boxb] (e1) {$a$};&
   \node[box] (e2) {\textcolor{myBlue}{$\tau_4$}}; &
    \node[float] (e3) {$1$};\\
    &&&
     \node[float] (f3) {$\shortparallel$};\\
      &&&
     \node[empty] (g3) {};\\
};
\node[fit=(a1) (e1), add paren] (matrix){};
\node (label) at ($(c0)!0!(c1)$) {{\Large $f$}};
\node (label) at ($(c1)!0.6!(c2)$) {:};
\draw[thick,shorten >=0cm,shorten <=0.1cm, |->] (c2)--(c3);
\node (label) at ($(c3)!0.6!(c4)$) {\hspace{5mm}\large{,}};

\node (label) at ($(z3)!0.3!(a3)$) {$\rotatebox{270}{\Bigg(}$};
\node (label) at ($(a3)!0.5!(b3)$) {\Large{$\cdot$}};
\node (label) at ($(b3)!0.6!(c3)$) {\Large{$\cdot$}};
\node (label) at ($(c3)!0.5!(d3)$) {\Large{$\cdot$}};
\node (label) at ($(d3)!0.5!(e3)$) {\Large{$\cdot$}};
\node (label) at ($(e3)!0.5!(f3)$) {$\rotatebox{270}{\Bigg)}$};
\node (label) at ($(f3)!1!(g3)$) {{$0$}};
\end{tikzpicture}
\end{subfigure}
   \hfill
  \begin{subfigure}[b]{0.3\textwidth}
\begin{tikzpicture}[
 add paren/.style={
      left delimiter={(},
      right delimiter={)}, 
    }]

\matrix(m) [column sep=0.5mm,   box/.style={
    shape=rectangle,
    text width=0.3cm,
    minimum height=0.6cm,
    text centered,
    draw=black,
    font=\large,
    font=\bf
},  
boxb/.style={
    shape=rectangle,
    text width=0.3cm,
    minimum height=0.6cm,
    text centered,
    draw=black,
    fill=lblue,
    font=\large,
    font=\bf
}, 
boxy/.style={
    shape=rectangle,
    text width=0.3cm,
    minimum height=0.6cm,
    text centered,
    draw=black,
    fill=lyellow,
    font=\large,
    font=\bf
}, 
boxh/.style={
    shape=rectangle,
    text width=0.3cm,
    minimum height=0.6cm,
    text centered,
    draw=black,
   pattern={mylines[angle=45,line width=2.5pt]}, 
   pattern color=lgray,
    font=\large,
    font=\bf
}, 
float/.style={
    shape=rectangle,
    text width=1.2cm,
    minimum height=0.6cm,
    text centered,
    outer sep=0.1cm,
},
empty/.style={
    shape=rectangle,
    text width=0cm,
    minimum height=0.6cm,
},] 
{
&
 \node[float] (z1) {$\eta$}; &
 \node[float] (z2) {$\tau''$};&
 \node[float] (z3) { };
  \\
  &
  \node[boxh] (a1) {$\eta_0$};&
   \node[box] (a2) {\textcolor{myGreen}{$\eta_0$}}; &
    \node[float] (a3) {$-w(a)$};\\
 &
  \node[boxh] (b1) {$\eta_1$};&
   \node[box] (b2) {\textcolor{BlueViolet}{$a$}}; &
    \node[float] (b3) {$w(\eta_1)$};\\
   \node[empty] (c0) { };&  
\node[boxh] (c1) {$\eta_2$};&
   \node[boxy] (c2) {\textcolor{BlueViolet}{$a$}}; &
    \node[float] (c3) {$w(\eta_2)$};&
       \node[empty] (c4) { }; \\
          &
\node[boxb] (d1) {$a$};&
   \node[box] (d2) {\textcolor{myBlue}{$\tau_3$}}; &
    \node[float] (d3) {$1$};\\
              &
\node[boxb] (e1) {$a$};&
   \node[box] (e2) {\textcolor{myBlue}{$\tau_4$}}; &
    \node[float] (e3) {$1$};\\
    &&&
     \node[float] (f3) {$\shortparallel$};\\
      &&&
     \node[empty] (g3) {};\\
};
\node[fit=(a1) (e1), add paren] (matrix){};
\node (label) at ($(c0)!0!(c1)$) {{\Large $f$}};
\node (label) at ($(c1)!0.6!(c2)$) {:};
\draw[thick,shorten >=0cm,shorten <=0.1cm, |->] (c2)--(c3);
\node (label) at ($(c3)!0.6!(c4)$) {\hspace{5mm}\large{,}};

\node (label) at ($(z3)!0.3!(a3)$) {$\rotatebox{270}{\Bigg(}$};
\node (label) at ($(a3)!0.5!(b3)$) {\Large{$\cdot$}};
\node (label) at ($(b3)!0.6!(c3)$) {\Large{$\cdot$}};
\node (label) at ($(c3)!0.5!(d3)$) {\Large{$\cdot$}};
\node (label) at ($(d3)!0.5!(e3)$) {\Large{$\cdot$}};
\node (label) at ($(e3)!0.5!(f3)$) {$\rotatebox{270}{\Bigg)}$};
\node (label) at ($(f3)!1!(g3)$) {{$-w(a) \cdot w(\eta_1)\cdot w(\eta_2)$}};
\end{tikzpicture}
\end{subfigure}
   \hfill
  \begin{subfigure}[b]{0.3\textwidth}
\begin{tikzpicture}[
 add paren/.style={
      left delimiter={(},
      right delimiter={)}, 
    }]

\matrix(m) [column sep=0.5mm,   box/.style={
    shape=rectangle,
    text width=0.3cm,
    minimum height=0.6cm,
    text centered,
    draw=black,
    font=\large,
    font=\bf
},  
boxb/.style={
    shape=rectangle,
    text width=0.3cm,
    minimum height=0.6cm,
    text centered,
    draw=black,
    fill=lblue,
    font=\large,
    font=\bf
}, 
boxy/.style={
    shape=rectangle,
    text width=0.3cm,
    minimum height=0.6cm,
    text centered,
    draw=black,
    fill=lyellow,
    font=\large,
    font=\bf
}, 
boxh/.style={
    shape=rectangle,
    text width=0.3cm,
    minimum height=0.6cm,
    text centered,
    draw=black,
   pattern={mylines[angle=45,line width=2.5pt]}, 
   pattern color=lgray,
    font=\large,
    font=\bf
}, 
float/.style={
    shape=rectangle,
    text width=1.2cm,
    minimum height=0.6cm,
    text centered,
    outer sep=0.1cm,
},
empty/.style={
    shape=rectangle,
    text width=0cm,
    minimum height=0.6cm,
},] 
{
&
 \node[float] (z1) {$\eta$}; &
 \node[float] (z2) {$\tau'''$};&
 \node[float] (z3) { };
  \\
  &
  \node[boxh] (a1) {$\eta_0$};&
   \node[box] (a2) {\textcolor{myGreen}{$\eta_0$}}; &
    \node[float] (a3) {$-w(a)$};\\
 &
  \node[boxh] (b1) {$\eta_1$};&
   \node[box] (b2) {\textcolor{BlueViolet}{$a$}}; &
    \node[float] (b3) {$w(\eta_1)$};\\
   \node[empty] (c0) { };&  
\node[boxh] (c1) {$\eta_2$};&
   \node[boxy] (c2) {\textcolor{myGreen}{$\eta_2$}}; &
    \node[float] (c3) {$-w(a)$};&
       \node[empty] (c4) { }; \\
          &
\node[boxb] (d1) {$a$};&
   \node[box] (d2) {\textcolor{myBlue}{$\tau_3$}}; &
    \node[float] (d3) {$1$};\\
              &
\node[boxb] (e1) {$a$};&
   \node[box] (e2) {\textcolor{myBlue}{$\tau_4$}}; &
    \node[float] (e3) {$1$};\\
    &&&
     \node[float] (f3) {$\shortparallel$};\\
      &&&
     \node[empty] (g3) {};\\
};
\node[fit=(a1) (e1), add paren] (matrix){};
\node (label) at ($(c0)!0!(c1)$) {{\Large $f$}};
\node (label) at ($(c1)!0.6!(c2)$) {:};
\draw[thick,shorten >=0cm,shorten <=0.1cm, |->] (c2)--(c3);

\node (label) at ($(z3)!0.3!(a3)$) {$\rotatebox{270}{\Bigg(}$};
\node (label) at ($(a3)!0.5!(b3)$) {\Large{$\cdot$}};
\node (label) at ($(b3)!0.6!(c3)$) {\Large{$\cdot$}};
\node (label) at ($(c3)!0.5!(d3)$) {\Large{$\cdot$}};
\node (label) at ($(d3)!0.5!(e3)$) {\Large{$\cdot$}};
\node (label) at ($(e3)!0.5!(f3)$) {$\rotatebox{270}{\Bigg)}$};
\node (label) at ($(f3)!1!(g3)$) {{$w(a)^2 \cdot w(\eta_1)$}};
\end{tikzpicture}
\end{subfigure}

\caption{Example of a cochain $f(\eta)$ applied to a sequence $\tau$ such that $\left( f(\eta) \right)(\tau)=0$, and of the same cochain $f(\eta)$ applied to sequences $\tau', \tau''$ and $\tau'''$ that have a single swapped vertex (highlighted in yellow) with $\tau$. Figured here, $\eta=(\eta_0,\eta_1,\eta_2,a,a)\in \X_4$ and $\tau=(\eta_0,\eta_1,\tau_2,\tau_3,\tau_4) \in \X_4$ with $\tau_2 \neq \eta_2$ (in red), such that $\big( f(\eta)\big) (\tau)=0$. In the second row we display the cochain $f(\eta)$ applied to $\tau',\tau'', \tau''' \bowtie \tau$: First, $\tau'$ such that $\tau'_2=\tau_2 \neq \eta_2$ (still in red), in this case it would also hold that $\big( f(\eta)\big) (\tau')=0$.  Then $\tau'' \& \tau'''$ such that $\tau_2'', \tau_2''' \neq \tau_2$ and specifically  $\tau_2'' = a$ and  $\tau_2''' = \eta_2$, in these cases we would have $\big( f(\eta)\big) (\tau''), \big( f(\eta)\big) (\tau''') \neq 0$.}\label{ex:eig zero}
\end{figure}

\begin{figure}[ht]
\centering
\begin{subfigure}[b]{0.3\textwidth}
\begin{tikzpicture}[
 add paren/.style={
      left delimiter={(},
      right delimiter={)}, 
    }]

\matrix(m) [column sep=0.5mm,   box/.style={
    shape=rectangle,
    text width=0.3cm,
    minimum height=0.6cm,
    text centered,
    draw=black,
    font=\large,
    font=\bf
},  
boxb/.style={
    shape=rectangle,
    text width=0.3cm,
    minimum height=0.6cm,
    text centered,
    draw=black,
    fill=lblue,
    font=\large,
    font=\bf
}, 
boxh/.style={
    shape=rectangle,
    text width=0.3cm,
    minimum height=0.6cm,
    text centered,
    draw=black,
   pattern={mylines[angle=45,line width=2.5pt]}, 
   pattern color=lgray,
    font=\large,
    font=\bf
}, 
float/.style={
    shape=rectangle,
    text width=1.2cm,
    minimum height=0.6cm,
    text centered,
    outer sep=0.1cm,
},
empty/.style={
    shape=rectangle,
    text width=0cm,
    minimum height=0.6cm,
},] 
{
&
 \node[float] (z1) {$\eta$}; &
 \node[float] (z2) {$\sigma$};&
 \node[float] (z3) { };
  \\
  &
  \node[boxh] (a1) {$\eta_0$};&
   \node[box] (a2) {\textcolor{myGreen}{$\eta_0$}}; &
    \node[float] (a3) {$-w(a)$};\\
 &
  \node[boxh] (b1) {$\eta_1$};&
   \node[box] (b2) {\textcolor{BlueViolet}{$a$}}; &
    \node[float] (b3) {$w(\eta_1)$};\\
   \node[empty] (c0) { };&  
\node[boxh] (c1) {$\eta_2$};&
   \node[box] (c2) {\textcolor{myGreen}{$\eta_2$}}; &
    \node[float] (c3) {$-w(a)$};&
       \node[empty] (c4) { }; \\
          &
\node[boxb] (d1) {$a$};&
   \node[box] (d2) {\textcolor{myBlue}{$\sigma_3$}}; &
    \node[float] (d3) {$1$};\\
              &
\node[boxb] (e1) {$a$};&
   \node[box] (e2) {\textcolor{myBlue}{$\sigma_4$}}; &
    \node[float] (e3) {$1$};\\
    &&&
     \node[float] (f3) {$\shortparallel$};\\
      &&&
     \node[empty] (g3) {};\\
};
\node[fit=(a1) (e1), add paren] (matrix){};
\node (label) at ($(c0)!0!(c1)$) {{\Large $f$}};
\node (label) at ($(c1)!0.6!(c2)$) {:};
\draw[thick,shorten >=0cm,shorten <=0.1cm, |->] (c2)--(c3);

\node (label) at ($(z3)!0.3!(a3)$) {$\rotatebox{270}{\Bigg(}$};
\node (label) at ($(a3)!0.5!(b3)$) {\Large{$\cdot$}};
\node (label) at ($(b3)!0.6!(c3)$) {\Large{$\cdot$}};
\node (label) at ($(c3)!0.5!(d3)$) {\Large{$\cdot$}};
\node (label) at ($(d3)!0.5!(e3)$) {\Large{$\cdot$}};
\node (label) at ($(e3)!0.5!(f3)$) {$\rotatebox{270}{\Bigg)}$};
\node (label) at ($(f3)!1!(g3)$) {{$w(a)^2 \cdot w(\eta_1)$}};
\end{tikzpicture}

\end{subfigure}

\begin{subfigure}[b]{0.3\textwidth}
\begin{tikzpicture}[
 add paren/.style={
      left delimiter={(},
      right delimiter={)}, 
    }]

\matrix(m) [column sep=0.5mm,   box/.style={
    shape=rectangle,
    text width=0.3cm,
    minimum height=0.6cm,
    text centered,
    draw=black,
    font=\large,
    font=\bf
},  
boxb/.style={
    shape=rectangle,
    text width=0.3cm,
    minimum height=0.6cm,
    text centered,
    draw=black,
    fill=lblue,
    font=\large,
    font=\bf
}, 
boxy/.style={
    shape=rectangle,
    text width=0.3cm,
    minimum height=0.6cm,
    text centered,
    draw=black,
    fill=lyellow,
    font=\large,
    font=\bf
}, 
boxh/.style={
    shape=rectangle,
    text width=0.3cm,
    minimum height=0.6cm,
    text centered,
    draw=black,
   pattern={mylines[angle=45,line width=2.5pt]}, 
   pattern color=lgray,
    font=\large,
    font=\bf
}, 
float/.style={
    shape=rectangle,
    text width=1.2cm,
    minimum height=0.6cm,
    text centered,
    outer sep=0.1cm,
},
empty/.style={
    shape=rectangle,
    text width=0cm,
    minimum height=0.6cm,
},] 
{
&
 \node[float] (z1) {$\eta$}; &
 \node[float] (z2) {$\sigma'$};&
 \node[float] (z3) { };
  \\
  &
  \node[boxh] (a1) {$\eta_0$};&
   \node[boxy] (a2) {\textcolor{Maroon}{$\sigma'_0$}}; &
    \node[float] (a3) {$0$};\\
 &
  \node[boxh] (b1) {$\eta_1$};&
   \node[box] (b2) {\textcolor{BlueViolet}{$a$}}; &
    \node[float] (b3) {$w(\eta_1)$};\\
   \node[empty] (c0) { };&  
\node[boxh] (c1) {$\eta_2$};&
   \node[box] (c2) {\textcolor{myGreen}{$\eta_2$}}; &
    \node[float] (c3) {$-w(a)$};&
       \node[empty] (c4) { }; \\
          &
\node[boxb] (d1) {$a$};&
   \node[box] (d2) {\textcolor{myBlue}{$\sigma_3$}}; &
    \node[float] (d3) {$1$};\\
              &
\node[boxb] (e1) {$a$};&
   \node[box] (e2) {\textcolor{myBlue}{$\sigma_4$}}; &
    \node[float] (e3) {$1$};\\
    &&&
     \node[float] (f3) {$\shortparallel$};\\
      &&&
     \node[empty] (g3) {};\\
};
\node[fit=(a1) (e1), add paren] (matrix){};
\node (label) at ($(c0)!0!(c1)$) {{\Large $f$}};
\node (label) at ($(c1)!0.6!(c2)$) {:};
\draw[thick,shorten >=0cm,shorten <=0.1cm, |->] (c2)--(c3);
\node (label) at ($(c3)!0.6!(c4)$) {\hspace{5mm}\large{,}};

\node (label) at ($(z3)!0.3!(a3)$) {$\rotatebox{270}{\Bigg(}$};
\node (label) at ($(a3)!0.5!(b3)$) {\Large{$\cdot$}};
\node (label) at ($(b3)!0.6!(c3)$) {\Large{$\cdot$}};
\node (label) at ($(c3)!0.5!(d3)$) {\Large{$\cdot$}};
\node (label) at ($(d3)!0.5!(e3)$) {\Large{$\cdot$}};
\node (label) at ($(e3)!0.5!(f3)$) {$\rotatebox{270}{\Bigg)}$};
\node (label) at ($(f3)!1!(g3)$) {{$0$}};
\end{tikzpicture}
\end{subfigure}
   \hfill
  \begin{subfigure}[b]{0.3\textwidth}
\begin{tikzpicture}[
 add paren/.style={
      left delimiter={(},
      right delimiter={)}, 
    }]

\matrix(m) [column sep=0.5mm,   box/.style={
    shape=rectangle,
    text width=0.3cm,
    minimum height=0.6cm,
    text centered,
    draw=black,
    font=\large,
    font=\bf
},  
boxb/.style={
    shape=rectangle,
    text width=0.3cm,
    minimum height=0.6cm,
    text centered,
    draw=black,
    fill=lblue,
    font=\large,
    font=\bf
}, 
boxy/.style={
    shape=rectangle,
    text width=0.3cm,
    minimum height=0.6cm,
    text centered,
    draw=black,
    fill=lyellow,
    font=\large,
    font=\bf
}, 
boxh/.style={
    shape=rectangle,
    text width=0.3cm,
    minimum height=0.6cm,
    text centered,
    draw=black,
   pattern={mylines[angle=45,line width=2.5pt]}, 
   pattern color=lgray,
    font=\large,
    font=\bf
}, 
float/.style={
    shape=rectangle,
    text width=1.2cm,
    minimum height=0.6cm,
    text centered,
    outer sep=0.1cm,
},
empty/.style={
    shape=rectangle,
    text width=0cm,
    minimum height=0.6cm,
},] 
{
&
 \node[float] (z1) {$\eta$}; &
 \node[float] (z2) {$\sigma''$};&
 \node[float] (z3) { };
  \\
  &
  \node[boxh] (a1) {$\eta_0$};&
   \node[boxy] (a2) {\textcolor{BlueViolet}{$a$}}; &
    \node[float] (a3) {$w(\eta_0)$};\\
 &
  \node[boxh] (b1) {$\eta_1$};&
   \node[box] (b2) {\textcolor{BlueViolet}{$a$}}; &
    \node[float] (b3) {$w(\eta_1)$};\\
   \node[empty] (c0) { };&  
\node[boxh] (c1) {$\eta_2$};&
   \node[box] (c2) {\textcolor{myGreen}{$\eta_2$}}; &
    \node[float] (c3) {$-w(a)$};&
       \node[empty] (c4) { }; \\
          &
\node[boxb] (d1) {$a$};&
   \node[box] (d2) {\textcolor{myBlue}{$\sigma_3$}}; &
    \node[float] (d3) {$1$};\\
              &
\node[boxb] (e1) {$a$};&
   \node[box] (e2) {\textcolor{myBlue}{$\sigma_4$}}; &
    \node[float] (e3) {$1$};\\
    &&&
     \node[float] (f3) {$\shortparallel$};\\
      &&&
     \node[empty] (g3) {};\\
};
\node[fit=(a1) (e1), add paren] (matrix){};
\node (label) at ($(c0)!0!(c1)$) {{\Large $f$}};
\node (label) at ($(c1)!0.6!(c2)$) {:};
\draw[thick,shorten >=0cm,shorten <=0.1cm, |->] (c2)--(c3);
\node (label) at ($(c3)!0.6!(c4)$) {\hspace{5mm}\large{,}};

\node (label) at ($(z3)!0.3!(a3)$) {$\rotatebox{270}{\Bigg(}$};
\node (label) at ($(a3)!0.5!(b3)$) {\Large{$\cdot$}};
\node (label) at ($(b3)!0.6!(c3)$) {\Large{$\cdot$}};
\node (label) at ($(c3)!0.5!(d3)$) {\Large{$\cdot$}};
\node (label) at ($(d3)!0.5!(e3)$) {\Large{$\cdot$}};
\node (label) at ($(e3)!0.5!(f3)$) {$\rotatebox{270}{\Bigg)}$};
\node (label) at ($(f3)!1!(g3)$) {{$-w(a) \cdot w(\eta_0)\cdot w(\eta_1)$}};
\end{tikzpicture}
\end{subfigure}
   \hfill
  \begin{subfigure}[b]{0.3\textwidth}
\begin{tikzpicture}[
 add paren/.style={
      left delimiter={(},
      right delimiter={)}, 
    }]

\matrix(m) [column sep=0.5mm,   box/.style={
    shape=rectangle,
    text width=0.3cm,
    minimum height=0.6cm,
    text centered,
    draw=black,
    font=\large,
    font=\bf
},  
boxb/.style={
    shape=rectangle,
    text width=0.3cm,
    minimum height=0.6cm,
    text centered,
    draw=black,
    fill=lblue,
    font=\large,
    font=\bf
}, 
boxy/.style={
    shape=rectangle,
    text width=0.3cm,
    minimum height=0.6cm,
    text centered,
    draw=black,
    fill=lyellow,
    font=\large,
    font=\bf
}, 
boxh/.style={
    shape=rectangle,
    text width=0.3cm,
    minimum height=0.6cm,
    text centered,
    draw=black,
   pattern={mylines[angle=45,line width=2.5pt]}, 
   pattern color=lgray,
    font=\large,
    font=\bf
}, 
float/.style={
    shape=rectangle,
    text width=1.2cm,
    minimum height=0.6cm,
    text centered,
    outer sep=0.1cm,
},
empty/.style={
    shape=rectangle,
    text width=0cm,
    minimum height=0.6cm,
},] 
{
&
 \node[float] (z1) {$\eta$}; &
 \node[float] (z2) {$\sigma'''$};&
 \node[float] (z3) { };
  \\
  &
  \node[boxh] (a1) {$\eta_0$};&
   \node[box] (a2) {\textcolor{myGreen}{$\eta_0$}}; &
    \node[float] (a3) {$-w(a)$};\\
 &
  \node[boxh] (b1) {$\eta_1$};&
   \node[boxy] (b2) {\textcolor{myGreen}{$\eta_1$}}; &
    \node[float] (b3) {$-w(a)$};\\
   \node[empty] (c0) { };&  
\node[boxh] (c1) {$\eta_2$};&
   \node[box] (c2) {\textcolor{myGreen}{$\eta_2$}}; &
    \node[float] (c3) {$-w(a)$};&
       \node[empty] (c4) { }; \\
          &
\node[boxb] (d1) {$a$};&
   \node[box] (d2) {\textcolor{myBlue}{$\sigma_3$}}; &
    \node[float] (d3) {$1$};\\
              &
\node[boxb] (e1) {$a$};&
   \node[box] (e2) {\textcolor{myBlue}{$\sigma_4$}}; &
    \node[float] (e3) {$1$};\\
    &&&
     \node[float] (f3) {$\shortparallel$};\\
      &&&
     \node[empty] (g3) {};\\
};
\node[fit=(a1) (e1), add paren] (matrix){};
\node (label) at ($(c0)!0!(c1)$) {{\Large $f$}};
\node (label) at ($(c1)!0.6!(c2)$) {:};
\draw[thick,shorten >=0cm,shorten <=0.1cm, |->] (c2)--(c3);

\node (label) at ($(z3)!0.3!(a3)$) {$\rotatebox{270}{\Bigg(}$};
\node (label) at ($(a3)!0.5!(b3)$) {\Large{$\cdot$}};
\node (label) at ($(b3)!0.6!(c3)$) {\Large{$\cdot$}};
\node (label) at ($(c3)!0.5!(d3)$) {\Large{$\cdot$}};
\node (label) at ($(d3)!0.5!(e3)$) {\Large{$\cdot$}};
\node (label) at ($(e3)!0.5!(f3)$) {$\rotatebox{270}{\Bigg)}$};
\node (label) at ($(f3)!1!(g3)$) {{$-w(a)^3$}};
\end{tikzpicture}
\end{subfigure}

\caption{Example of a cochain $f(\eta)$ applied to a sequence $\tau$ such that $\big( f(\eta)\big) (\sigma) \neq 0$, and of the same cochain $f(\eta)$ applied to sequences $\sigma', \sigma''$ and $\sigma'''$ that have a single swapped vertex (highlighted in yellow) with $\sigma$.
Figured here $\eta=(\eta_0,\eta_1,\eta_2,a,a)\in \X_4$ and $\sigma=(\eta_0,\eta_1,\eta_2,\sigma_3,\sigma_4) \in \X_4$ such that $\big( f(\eta)\big) (\sigma)=w(a)^2 \cdot w(\eta_1)$. In the second row,  we display the cochain $f(\eta)$ applied to $\sigma',\sigma'', \sigma''' \bowtie \sigma$: First, $\sigma'$ such that $\sigma'_0 \neq \eta_0, a$ (now in red), in this case it would hold that $\big( f(\eta)\big) (\sigma')=0$.  Then $\sigma'' \& \sigma'''$ such that   $\sigma_0 \neq \sigma_0'' = a$ and  $\sigma_1 \neq \sigma_1''' = \eta_1$, in these cases we would have $\big( f(\eta)\big) (\sigma''), \big( f(\eta)\big) (\sigma''') \neq 0 ,  \big( f(\eta)\big) (\sigma)$.}\label{ex:eig nonzero}
\end{figure}

As stated, we will begin by showing for all $n \in \ZZ_{\geq 0}$ the cochain, $f({(a, \dots , a)}) \in C^n(\X)$,  is an eigenvector. Specifically, that it is associated to eigenvalue $\lambda=1$.  By definition $f_0(a)(y)=1$ for all $y \in V$,  hence 
$f({(a, \dots , a)})= f_0(a) \otimes \dots \otimes f_0(a)=
\sum_{\sigma \in \X_n} e_\sigma\in C^n(\X)$ 
is the constant one cochain,  we will refer to it as $\mathds{1} := f({(a, \dots , a)})$.  For any $\tau \in \X_n$, we may take $v=\mathds{1}$ on the right hand side of (\ref{eigvec}), and get:

\begingroup
\addtolength{\jot}{1em}
\begin{align}
&\left[ (n+2) -\sum_{i=0}^n w(\tau_i) \right] \cdot \mathds{1}(\tau)
- \sum_{\substack{\sigma :\sigma \bowtie \tau,\\ \tau_j \neq \sigma_j}} w(\sigma_j) \cdot \mathds{1}({\sigma}) 
=(n+2) -\sum_{i=0}^n w(\tau_i) 
- \sum_{\substack{\sigma :\sigma \bowtie \tau,\\ \tau_j \neq \sigma_j}} w(\sigma_j) \nonumber\\
=&(n+2) -\sum_{i=0}^n w(\tau_i) 
- \sum_{j=0}^n \sum_{u \neq \tau_j} w(u) \nonumber
=(n+2) 
- \sum_{i=0}^n \sum_{u \in V} w(u) \nonumber
=(n+2)-(n+1) 
=1 \nonumber
\end{align}
\endgroup
\noindent Where  the second to last equality was due to $\sum_{u \in V} w(u)=1$.  We conclude that ,
\begin{align*}
1 \cdot \mathds{1}(\tau) =1 =\left[ (n+2) -\sum_{i=0}^n w(\tau_i) \right] \cdot \mathds{1}(\tau)
- \sum_{\substack{\sigma :\sigma \bowtie \tau,\\ \tau_j \neq \sigma_j}} w(\sigma_j) \cdot \mathds{1}({\sigma}) \hspace{3mm} \text{ ,  } \forall \tau
\end{align*}
Proving $f({(a, \dots , a)})=\mathds{1}$ is an eigenvector corresponding to $\lambda=1$. \\

We will now show that for any $k \in [0:n]$, and $\eta \in  \X_n$ such that $\# \{a \in \eta \} =n-k $,  the cochain $f(\eta)$  as defined in Definition \ref{def:eig basis}, is an eigenvector associated to $\lambda=k+2$.\\

Fix $k \in \{ 0,...,n \}$ and $\eta \in \X_n$ such that $\# \{a \in \eta \} =n-k $. Denote by $I_a \subseteq [0:n]$ the set  of indices such that $\eta_i=a$ for all $i\in I_a$.As we mentioned earlier, we first deal with sequences $\tau \in \X_n$ such that $\big( f(\eta)\big)(\tau)=0$. Consider  any sequence $\tau$ that admits $ i \not\in I_a$, such that the vertex $\tau_i \neq a, \eta_i$,  by definition,  it holds that 
$$\big(f_0(\eta_i)\big)(\tau_i)=0 \qquad \Longrightarrow \qquad \big(f(\eta)\big)(\tau) =0$$

Fix such a $\tau$ and let $i^* \in [0:n]\setminus I_a$ be such that $\tau_{i^*} \neq a, \eta_{i^*}$, again we look at the right hand side of equation (\ref{eigvec}):

As $\big(f(\eta)\big)(\tau)=0$, we immediately have:
\begin{equation}\label{zero case}
    \left[ (n+2) -\sum_{i=0}^n w(\tau_i) \right] \cdot \big(f(\eta)\big)(\tau)
- \sum_{\substack{\sigma :\sigma \bowtie \tau,\\ \tau_j \neq \sigma_j}} w(\sigma_j) \cdot \big(f(\eta)\big)(\sigma) = - \sum_{\substack{\sigma :\sigma \bowtie \tau,\\ \tau_j \neq \sigma_j}} w(\sigma_j) \cdot \big(f(\eta)\big)(\sigma)
\end{equation}

Now,  for $\sigma \bowtie \tau$ such that $\sigma_j \neq \tau_j$ and $j \neq i^*$, we would still have $ \big(f(\eta)\big)(\sigma)=0$. 
Moreover,  if $j=i^*$ but $\sigma_{i^*} \neq a, \eta_{i^*}$, we would also still have that $ \big(f(\eta)\big)(\sigma)=0$.\\
Hence the sum (\ref{zero case}) reduces to:
\begin{align}\label{zero case 2}
 \sum_{\substack{\sigma :\sigma \bowtie \tau,\\ \tau_j \neq \sigma_j}} w(\sigma_j) \cdot  \big(f(\eta)\big)(\sigma) &= \sum_{\substack{\sigma :\sigma \bowtie \tau,\\ \sigma_{i^*}=a}} w(\sigma_{i^*}) \cdot  \big(f(\eta)\big)(\sigma) + \sum_{\substack{\sigma :\sigma \bowtie \tau,\\  \sigma_{i^*}={\eta_{i^*}}}} w(\sigma_{i^*}) \cdot  \big(f(\eta)\big)(\sigma)
\end{align}

Furthermore, as there exists only one $\sigma$ such that $\tau \bowtie \sigma$ and $\sigma_{i^*}=a$, and one $\sigma$ such that $\tau \bowtie \sigma$ and $\sigma_{i^*}=\eta_{i^*}$,  we can reduce the sum (\ref{zero case 2}) to the sum of exactly two terms:

\begingroup
\addtolength{\jot}{0.2em}
\begin{align*}
& \sum_{\substack{\sigma :\sigma \bowtie \tau,\\ \sigma_{i^*}=a}} w(\sigma_{i^*}) \cdot  \big(f(\eta)\big)(\sigma) + \sum_{\substack{\sigma :\sigma \bowtie \tau,\\  \sigma_{i^*}={\eta_{i^*}}}} w(\sigma_{i^*}) \cdot  \big(f(\eta)\big)(\sigma)\\
 =&w(a) \cdot \Biggl( (w(\eta_{i^*})) \cdot \prod_{\substack{s \in [0:n],\\s \neq i^*}} \big(f_0({\eta_s})\big)(\tau_s)\Biggr) +  w(\eta_{i^*}) \cdot \Biggl( (-w(a)) \cdot \prod_{\substack{s \in [0:n],\\s \neq i^*}} \big(f_0({\eta_s})\big)(\tau_s) \Biggr)\\
=&w(a) \cdot w(\eta_i) \cdot \prod_{\substack{s \in [0:n],\\s \neq i^*}}  \big(f_0({\eta_s})\big)(\tau_s) \cdot [1-1] =0 
\end{align*}
\endgroup

We conclude that, for $\tau$ that admit an $ i \in [0:n]\setminus I_a$ such that the vertex $\tau_i \neq a, \eta_i$: 

$$\lambda \cdot \big(f(\eta)\big)(\tau)= 0 = \left[ (n+2) -\sum_{i=0}^n w(\tau_i) \right] \cdot \big(f(\eta)\big)(\tau)
- \sum_{\substack{\sigma :\sigma \bowtie \tau,\\ \tau_j \neq \sigma_j}} w(\sigma_j) \cdot \big(f(\eta)\big)(\sigma)$$

Proving the cochain $f(\eta)$ satisfies the system of equations corresponding to the eigenvalue $\lambda$, on sequences $\tau$ with $\tau_i \neq a, \eta_i$ for $i \in [0:n]\setminus I_a$.  We will now show  the system is also satisfied for any other $\tau \in \X_n$.   In this case for all $i \in [0:n]\setminus I_a$ the vertices $\tau_i$ are forced to be either $a,\eta_i$. Hence we deal as promised with the case of $\tau \in \X_n$, such that $\big(f(\eta)\big)(\tau)=0$.

We use the swap operation from Definition \ref{defs} to rewrite the sum:
\begin{align}\label{sum}
\sum_{\substack{\sigma :\sigma \bowtie \tau,\\ \tau_j \neq \sigma_j}} w(\sigma_j) \cdot \big(f(\eta)\big)(\sigma)= \sum_{j=0}^n \sum_{y \neq \tau_j} w(y) \cdot \big(f(\eta)\big)({\swap{\tau}{j}{y}})
\end{align}

We have assumed, $\eta_j=a$  for all $j\in I_a$.  As such, for all $y \in V$:
\begin{align*}
 \big(f(\eta)\big)({\swap{\tau}{j}{y}})=   \big(f(\eta)\big)(\tau) \hspace{5mm} \forall j \in I_a\\
\end{align*}

Meaning we can reduce part of the sum in (\ref{sum}), when restricted to $j \in I_a$, to:
\begingroup
\addtolength{\jot}{0.5em}
\begin{align}\label{sum on j gives a}
 \sum_{j \in I_a} \sum_{y \neq \tau_j} w(y) \cdot \big(f(\eta)\big)({\swap{\tau}{j}{y}})
= \sum_{j \in I_a} \sum_{y \neq \tau_j} w(y) \cdot  \big(f(\eta)\big)(\tau)
= \sum_{j \in I_a} \left(1 -w(\tau_j)  \right) \cdot \big(f(\eta)\big)(\tau)
\end{align}
\endgroup
Notice in the last line we used the fact that $\sum_{y \in V} w(y)=1$. 

On the other hand,  for $j \in [0:n]\setminus I_a$ we have:
\begin{align*}
 \big(f(\eta)\big)({\swap{\tau}{j}{y}})&=  \big(f_0(\eta_j)\big)(y) \cdot \prod_{s \neq j}  \big(f_0(\eta_s)\big)(\tau_s)
\end{align*}
In particular, as we have assumed $\eta_j \neq a$,  for $j \not\in I_a$,  this means 
\begin{align*}
 \big(f(\eta)\big)({\swap{\tau}{j}{y}})&=
\begin{cases}
\big(f_0(\eta_j)\big)(y) \cdot \displaystyle\prod_{s \neq j}  \big(f_0(\eta_s)\big)(\tau_s) \neq 0 \hspace{3mm}  \text{if }  y=\eta_j,a\\
0 \hspace{50mm}  \text{ if }  y \neq \eta_j,a
\end{cases}
\end{align*}

We recall that for $j \in [0:n]\setminus I_a$ the vertices $\tau_j$ are forced to be either $a,\eta_j$.   Consider the case $\tau_j=a$,  then we can reduce part of the sum in (\ref{sum}) to:
\begingroup
\addtolength{\jot}{0.5em}
\begin{align*}
&\sum_{y \neq \tau_j} w(y) \cdot \big(f(\eta)\big)({\swap{\tau}{j}{y}}) 
=w(\eta_j) \cdot  \big(f_0(\eta_j)\big)(\eta_j) \cdot \displaystyle\prod_{s \neq j} \big(f_0(\eta_s)\big)(\tau_s)\\
=& w(\eta_j) \cdot  (-w(a)) \cdot \displaystyle\prod_{s \neq j}  \big(f_0(\eta_s)\big)(\tau_s)
=  \big(f_0(\eta_j)\big)(\tau_j) \cdot (-w(a)) \cdot \displaystyle\prod_{s \neq j}  \big(f_0(\eta_s)\big)(\tau_s)
= -w(\tau_j) \cdot \big(f(\eta)\big)(\tau)
\end{align*}
\endgroup
Similarly,  consider the case where $\tau_j=\eta_j$,  again we can reduce the sum to:
\begingroup
\addtolength{\jot}{0.5em}
\begin{align*}
&\sum_{y \neq \tau_j} w(y) \cdot \big(f(\eta)\big)({\swap{\tau}{j}{y}}) 
=w(a) \cdot  \big(f_0(\eta_j)\big)(a) \cdot \displaystyle\prod_{s \neq j}  \big(f_0(\eta_s)\big)(\tau_s)\\
=& w(a) \cdot  w(\eta_j) \cdot \displaystyle\prod_{s \neq j}  \big(f_0(\eta_s)\big)(\tau_s)
= -  \big(f_0(\eta_j)\big)(\tau_j) \cdot  w(\eta_j) \cdot \displaystyle\prod_{s \neq j}  \big(f_0(\eta_s)\big)(\tau_s)
= -w(\tau_j) \cdot \big(f(\eta)\big)(\tau)
\end{align*}
\endgroup

Putting both these cases together we can rewrite the  sum in (\ref{sum}) restricted to $[0:n]\setminus I_a$, as:
\begin{align}\label{sum j not a}
 \sum_{j \not\in I_a} \sum_{y \neq \tau_j} w(y) \cdot \big(f(\eta)\big)({\swap{\tau}{j}{y}})
&= \sum_{j \not\in I_a} -w(\tau_j) \cdot \big(f(\eta)\big)(\tau)
\end{align}
 Putting the restricted sums (\ref{sum on j gives a}) and (\ref{sum j not a}) together, and plugging into the right hand side of (\ref{eigvec}),  we see that for $\tau$ for which all vertices $\tau_i$, for $i \not\in I_a$,  are forced to be either $a,\eta_i$, we have:
\begingroup
\addtolength{\jot}{0.5em}
\begin{align*}
&\left[ (n+2) -\sum_{j=0}^n w(\tau_j) \right] \cdot \big(f(\eta)\big)(\tau)
- \sum_{\substack{\sigma :\sigma \bowtie \tau,\\ \tau_j \neq \sigma_j}} w(\sigma_j) \cdot\big(f(\eta)\big)({\sigma})\\
= &\left[ (n+2) -\sum_{j=0}^n w(\tau_j) \right] \cdot \big(f(\eta)\big)(\tau) +
 \sum_{j \not\in I_a} w(\tau_j) \cdot \big(f(\eta)\big)(\tau) - \sum_{j \in I_a} \left(1 -w(\tau_j)  \right) \cdot \big(f(\eta)\big)(\tau) 
 = (k+2)\cdot \big(f(\eta)\big)(\tau)   
\end{align*}
\endgroup

Notice in the last equality we used $|I_a|=n-k$, as per our assumption. With this we conclude $f(\eta)$ satisfies the system of equations corresponding to the eigenvalue $\lambda=k+2$ for all sequences $\tau \in \X_n$. Proving $f(\eta)$ is an eigenvector for $\lambda=k+2$ for $\eta$ such that $\#  \{ a \in \eta \} =n-k$.\\ 

As we know the cochains $f(\eta)$ form a basis for $C^n$, they must span the eigenspaces of the Laplacian, $L_n$. Hence the multiplicity of each eigenvalue $\lambda=1,...,n+2$ is easily computed to be,
\begin{align*}
\text{mult}(\lambda)&= \# \Big\{ \eta \in \X_n : \# \{a \in \eta \} = n+2-\lambda \Big\} 
=\binom{n+1}{\lambda-1} \cdot (|V|-1)^{\lambda-1}
\end{align*}

With this, we finish the proof of the main theorem.
\end{proof}

\medskip
\section{Proof of Theorem \ref{prop:IndSimpLap} for Simplicial Complexes}\label{sec:Simpl cpx}
In this section we prove  the analogous results to the previous section's, in the setting of simplicial complexes. As we mentioned earlier, the results in this section, if at all new, are simple consequences of classical results for the Hodge Laplacian. We present here results, old and new, for completion, and to provide a parallel to the sequence complex setting.

\subsection{Laplacian over Simplicial Complexes}
We will consider a weighted simplicial complex $K \subseteq 2^{[m]}$  with weight function $w:  K\to(0,+\infty)$ and associated inner product  $\langle \cdot , \cdot \rangle$ on $C^*(K)$, defined via its values on the usual basis$\{e_{\xi}\}$ as: 
 \begin{equation*}
\langle e_{\xi}, e_{\xi'}\rangle \od 
\begin{cases}
w(\xi) , & \text{ if } \xi= \xi' \\
0& \text{ if } \xi \neq  \xi'
\end{cases}
 \end{equation*} 

In this setting the coboundary $\d$ and coboundary adjoint $\d^*$ operators are defined for any face $\xi =\{i_0,...,i_{n+1} \}  \in K_{n+1}$,  with vertices labeled such that $i_0< \dots < i_{n+1}$, as:
\begin{align*}
  \d_n e_{\xi }  
&= \sum_{\substack{i \in [m]\setminus \xi,\\ \{i\}\cup\xi \in K}} \kappa({\{i\}\cup\xi},{\xi}) \cdot e_{\{i\}\cup\xi }\\
\d_n^* e_{\xi}  
&= \sum_{\xi' \in K_n}  \frac{w(\xi)}{w(\xi')}  \kappa({\xi},{\xi'}) e_{\xi' }
= \sum_{j=0}^{n+1}  \frac{w(\xi)}{w(\xi \setminus \{i_j\})}  (-1)^j e_{\xi \setminus \{i_j \} }
\end{align*}
Where the incidence function $\kappa$ is defined as following Definition \ref{ex:simplicial}. In this setting, the Laplacian is easily computed as stated in the following lemma.\\

\begin{lemma}\label{prop:simpLapl}
Let $2^{[m]}$ be  a weighted full simplex, with weight function $w:2^{[m]} \to (0,1]$. For all $n \in \ZZ_{\geq -1}$ and   faces $\xi =\{i_0,...,i_{n} \}  \in K_{n}$,   with vertices labeled such that $i_0< \dots < i_n$,  the Laplacian operator on the basis elements $e_{\xi }\in C^n(2^{[m]})$,  is given by:
  \begingroup
\addtolength{\jot}{0.8em}
\begin{align*}
L_n(e_{\xi })= &\left( 
\sum_{i \in [m] \setminus \xi}  \dfrac{w(\{i\} \cup \xi)}{w(\xi)} + \sum_{k=0}^n \dfrac{w(\xi)}{w(\xi \setminus \{i_k\})}
\right) \cdot e_{\xi }\\
&+ \sum_{i \in [m] \setminus \xi }  \sum_{k=0}^{n} 
(-1)^k \cdot  \kappa(\{i\} \cup \xi \setminus \{i_k\},\xi \setminus \{i_k\}) 
\cdot \Bigg(
\frac{w(\xi)}{w(\xi \setminus \{i_k\})}
- \frac{w(\{i\} \cup \xi)}{w((\{i\} \cup \xi) \setminus \{i_k\})}
\Bigg)  \cdot  e_{\{i\}\cup\xi \setminus \{i_k\} }
\end{align*}
\endgroup
\end{lemma}

\begin{proof}
We compute $L_n^{down}$ and $L_n^{up}$ separately,  by plugging in our computations for  $\d$ and $\d^*$ each time. Consider a face $\xi =\{i_0,...,i_{n} \}  \in K_{n}$,  with vertices labeled such that $i_0< \dots < i_n$:
  \begingroup
\addtolength{\jot}{1em}
\begin{align}\label{calc:Lup}
L_n^{up} e_{\xi } =& \d^* \d_n e_{\xi }
= \d^* \left( \sum_{i \in [m]\setminus \xi} \kappa({\{i\}\cup\xi},{\xi}) \cdot e_{\{i\}\cup\xi } \right)\nonumber\\
=& \sum_{i \in [m]\setminus \xi} \kappa({\{i\}\cup\xi},{\xi}) \cdot \Bigg[
\sum_{j \in \{i\} \cup \xi}  \frac{w(\{i\} \cup \xi)}{w((\{i\} \cup \xi) \setminus \{j\})} \kappa(\{i\} \cup \xi, \{i\} \cup \xi \setminus \{j\}) \cdot e_{\{i\} \cup \xi \setminus \{j\}} \Bigg] \nonumber\\
=& \sum_{i \in [m]\setminus \xi} \kappa({\{i\}\cup\xi},{\xi}) \cdot \Bigg[
\frac{w(\{i\} \cup \xi)}{w(\xi)}\kappa({\{i\}\cup\xi},{\xi})  \cdot  e_{\xi } \nonumber\\ 
 &+ \sum_{\substack{i_k \in \xi, \\ i_k < i}} \frac{w(\{i\} \cup \xi)}{w((\{i\} \cup \xi) \setminus \{i_k\})}  (-1)^{k} \cdot e_{\{i\}\cup\xi \setminus \{i_k\} }
 + \sum_{\substack{i_k \in \xi, \\ i_k > i}}  \frac{w(\{i\} \cup \xi)}{w((\{i\} \cup \xi) \setminus \{i_k\})}  (-1)^{k+1} \cdot e_{\{i\}\cup\xi \setminus \{i_k\} } \Bigg]
\end{align}

Here we made use of the simple observation:
\begin{align*}
\kappa(\{i\} \cup \xi, \{i\} \cup \xi \setminus \{i_k\})= 
\begin{cases}
(-1)^k \hspace{5mm} &\text{if } i_k<i\\
(-1)^{k+1} \hspace{5mm} &\text{if } i_k>i
\end{cases}
\end{align*}
On the other hand, notice:
\begin{align*}
\kappa(\{i\} \cup \xi \setminus \{i_k\},   \xi \setminus \{i_k\})= 
\begin{cases}
-\kappa(\{i\} \cup \xi,   \xi ) \hspace{5mm} &\text{if } i_k<i\\
\kappa(\{i\} \cup \xi,   \xi ) \hspace{5mm} &\text{if } i_k>i
\end{cases}
\end{align*}
We can use this, together with $\Big( \kappa(\{i\} \cup \xi,   \xi \Big)^2=1$,  to simplify the expression from (\ref{calc:Lup}) :
\begin{align*}\label{calc:Lup}
L_n^{up} e_{\xi } 
= \sum_{i \in [m]\setminus \xi}
\frac{w(\{i\} \cup \xi)}{w(\xi)}   \cdot  e_{\xi } 
 + \sum_{i \in [m]\setminus \xi}  \sum_{k=0}^n 
 \frac{w(\{i\} \cup \xi)}{w((\{i\} \cup \xi) \setminus \{i_k\})}
 \kappa(\{i\} \cup \xi \setminus \{i_k\},   \xi \setminus \{i_k\}) 
   \cdot (-1)^{k+1} \cdot e_{\{i\}\cup\xi \setminus \{i_k\} }
\end{align*}
Now let us look at $L_n^{down}$,
\begin{align*}
&L_n^{down} e_{\xi }= \d \d_{n-1}^* e_{\xi }
= \d \left( \sum_{k=0}^{n}  \frac{w(\xi)}{w(\xi \setminus \{i_k\})}  (-1)^k \cdot e_{\xi \setminus \{i_k \} } \right)\\
=&\sum_{k=0}^{n} 
\frac{w(\xi)}{w(\xi \setminus \{i_k\})}  (-1)^k
\Bigg[
  \kappa(\xi,\xi \setminus \{i_k\}) \cdot e_{\xi}
+ \sum_{i \in [m] \setminus \xi }
 \kappa(\{i\} \cup \xi \setminus \{i_k\},\xi \setminus \{i_k\}) \cdot e_{\{i\}\cup\xi \setminus \{i_k\} }
\Bigg]\\
=&\sum_{k=0}^{n} 
\frac{w(\xi)}{w(\xi \setminus \{i_k\})}  (-1)^{2k} \cdot e_{\xi}
+ \sum_{i \in [m] \setminus \xi }  \sum_{k=0}^{n} 
\frac{w(\xi)}{w(\xi \setminus \{i_k\})}
(-1)^k \cdot
 \kappa(\{i\} \cup \xi \setminus \{i_k\},\xi \setminus \{i_k\}) 
 \cdot  e_{\{i\}\cup\xi \setminus \{i_k\} }
\end{align*}
\endgroup

Putting $L_n^{down}$ and $L_n^{up}$ together, we get 
  \begingroup
\addtolength{\jot}{1em}
\begin{align*}
L_n(e_{\xi })=&L^\down_n(e_{\xi }) + L^\up_n(e_{\xi })
= \left( 
\sum_{i \in [m] \setminus \xi}   \dfrac{w(\{i\} \cup \xi)}{w(\xi)} + \sum_{k=0}^n \dfrac{w(\xi)}{w(\xi \setminus \{i_k\})}
\right) \cdot e_{\xi }\\
&+ \sum_{i \in [m] \setminus \xi }  \sum_{k=0}^{n} 
(-1)^k \cdot  \kappa(\{i\} \cup \xi \setminus \{i_k\},\xi \setminus \{i_k\}) 
\cdot \Bigg(
\frac{w(\xi)}{w(\xi \setminus \{i_k\})}
- \frac{w(\{i\} \cup \xi)}{w((\{i\} \cup \xi) \setminus \{i_k\})}
\Bigg)  \cdot  e_{\{i\}\cup\xi \setminus \{i_k\} }
\end{align*}
\endgroup
Proving the expression in the proposition statement.
\end{proof}

An immediate consequence of this lemma, tells us that scaling the weight function, by the same factor over all dimensions, has no impact on the associated Laplacian.

\begin{lemma}\label{lemma:scaling}
Let $2^{[m]}$ be  a weighted full simplex, with weight function $w:2^{[m]} \to (0,1]$. Consider a different weight function $w':2^{[m]} \to (0,1]$,  defined as $w'(\xi)=\alpha \cdot w(\xi)$ for all $\xi \in 2^{[m]}$, for some $\alpha \in \R$. The respective associated Laplacians $L^w_n$ and $L^{w'}_n$ satisfy $$L^w_n=L^{w'}_n \quad \text{ for all } n \geq -1.$$
\end{lemma}

It is also sometimes useful to think of the matrix representation of $L_n$ directly in terms of the matrix representation of $\d$. This particularly comes in useful when relating to the combinatorial Laplacian. Let $D_n$ be the matrix representation of $\d_n$ and $W_n$ the diagonal matrix representing the inner product on $C^n(K)$. Then we may express the $L_n^\up, L_n^\down$ operators as:
\begin{align*}
    L_n^\up &=W_n^{-1} D_n^T W_{n+1}D_n\\
    L_n^\down &=D_{n-1}W_{n-1}^{-1}D_{n-1}^T W_n
\end{align*}

Consider $K=G$ to be a weighted graph. Denote by $\mathbb{A}$ the diagonal matrix of vertex weights, following our definition $\mathbb{A}=W_0$. On the other hand, let $\mathbb{W}$ be the \textit{weighted adjacency matrix}: $\mathbb{W}_{ij} = w(\{i,j\})$, and $\mathbb{D}$ be the \textit{degree matrix} given by, $\mathbb{D}_{ii}=\sum_{j \neq i} \mathbb{W}_{ij}$ . It is easy to show that, $$\mathbb{D} - \mathbb{W} = D_0^T W_1 D_0.$$
Which implies 
$$L_0^\up =W_0^{-1} D_0^T W_{1}D_0 = \mathbb{A}^{-1} [\mathbb{D} - \mathbb{W}]= L_{\text{comb}}$$
As we had referenced previously in Section \ref{subsec:hodge}.\\

\medskip
\subsection{Laplacian on Simplicial Complexes for the  Independent Vertices Model}
Here we  now show the Laplacian for the independent vertices  model is ``boring"; we  recall the theorem we stated in section \ref{results}:

\begin{reptheorem}{prop:IndSimpLap}
Let $(2^{[m]},w)$ be a weighted simplicial complex, with  weight function $w:2^{[m]} \to \R_{>0}$ such that $w(\varnothing)=1$.  Then the following two conditions are equivalent:
\begin{enumerate}
\item[(i)] $(2^{[m]},w)$ is an independent vertices model;
\item[(ii)] the associated Laplacians are multiples of identities:
\begin{equation} \label{eq:identities}
L_n=\alpha_n \cdot I_{C^n}, \quad \forall n\geq -1,   
\end{equation}
where $I_{C^n}$ denotes  the identity on $C^n(2^{[m]})$, and $\alpha_n \in \R$.
\end{enumerate} 

Moreover, if \eqref{eq:identities} holds, then  $\alpha_n$ is constant with respect to  $n$, and  
\begin{equation*} \alpha_n = \sum_{i=1}^m w_i,\end{equation*}
where  $(w_1,\dots,w_m)$ is the vector determining the independent vertices model associated to the weight function $w$.\\
\end{reptheorem}

\begin{proof}
We start by proving the reverse direction: we assume our simplicial complex follows the  independent weights model,  and wish to prove the associated Laplacian $L_n=\alpha_n \cdot I_{C^n(2^{[m]})}$.  Notice that for $n=-1,m-1$, we have $|(2^{[m]})_n|=1$,  making the statement trivial.  We now focus on the cases for $n=0,...,m-1$: Fix $e_\xi \in C^n(2^{[m]})$, we call upon Lemma \ref{prop:simpLapl}:
  \begingroup
\addtolength{\jot}{1em}
\begin{align*}
&L_n(e_{\xi })
= \sum_{i \in [m] \setminus \xi }  \sum_{k=0}^{n} 
(-1)^k \cdot  \kappa(\{i\} \cup \xi \setminus \{i_k\},\xi \setminus \{i_k\}) 
\cdot \Bigg(
\frac{w(\xi)}{w(\xi \setminus \{i_k\})}
- \frac{w(\{i\} \cup \xi)}{w((\{i\} \cup \xi) \setminus \{i_k\})}
\Bigg)  \cdot  e_{\{i\}\cup\xi \setminus \{i_k\} }\\
+&\left( 
\sum_{i \in [m] \setminus \xi}   \dfrac{w(\{i\} \cup \xi)}{w(\xi)} + \sum_{k=0}^n \dfrac{w(\xi)}{w(\xi \setminus \{i_k\})}
\right) \cdot e_{\xi }
\hspace{2mm}=\hspace{2mm} 
\left( 
\sum_{i \in [m] \setminus \xi}  w(i) + \sum_{k=0}^n w(i_k)
\right) \cdot e_{\xi }\\
+& \sum_{i \in [m] \setminus \xi }  \sum_{k=0}^{n} 
(-1)^k \cdot  \kappa(\{i\} \cup \xi \setminus \{i_k\},\xi \setminus \{i_k\}) 
\cdot \Big( w(i_k) - w(i_k) \Big)  
\cdot  e_{\{i\}\cup\xi \setminus \{i_k\} }
\hspace{2mm}=\hspace{2mm}  
\left( 
\sum_{i \in [m] }   w(i) 
\right) \cdot e_{\xi }
+ 0\\
\end{align*}
\endgroup
Proving $L_n=\alpha_n \cdot I_{C^n(2^{[m]})}$ for $\alpha_n=\sum_{i \in [m]}w(i)$ for all $n= 0,...,m-2$.  In the cases of $n=-1,m-1$, the computation of $L_n$ easily shows that $\alpha_n$ must also be $\alpha_n=\sum_{i \in [m]}w(i)$.

We now proceed to prove the forward direction, we assume for all $n=-1,\dots,m$, it holds that $L_n=\alpha_n \cdot I_{C^n(2^{[m]})}$.  Putting this assumption together with Lemma \ref{prop:simpLapl} tells us that for all $n \in [-1:m]$, face $\xi \in \left(2^{[m]}\right)_n$, and vertices $i \in [m]\setminus \xi, i_k \in \xi$ it holds that 
$$ \dfrac{w(\{i\} \cup \xi)}{w(\{i\} \cup \xi \setminus \{i_k\})} =\dfrac{w(\xi)}{w(\xi \setminus \{i_k\})}$$

\newpage
We will use this consequence of our hypothesis to prove the weights follow the independent vertices model by applying induction over $n$.  The statement is assumed in the proposition statement for $n=-1$ and vacuous for $n=0$, let us set $n=1$.
Consider $\{i,j\} \in \left(2^{[m]}\right)_1$,  by hypothesis:
\begin{align*}
 \dfrac{w(\{i\} \cup \{j\})}{w(\{i\} \cup \{j\} \setminus \{j\})} =\dfrac{w(\{j\})}{w(\{j\} \setminus \{j\})}
 \hspace{2mm} \iff \hspace{2mm}
  \dfrac{w(\{i,j\})}{w(\{i\} )} =\dfrac{w(\{j\})}{w(\varnothing)}
\end{align*}
As the proposition statement assumes $w(\varnothing)=1$, we conclude $w(\{i,j\})=w(\{i\})w(\{j\})$, proving the base case. Now assume the result holds for faces of dimension $n$,  we will show it holds for faces of dimension $n+1$.\\
Consider $\{i_0,\dots,i_{n+1}\} \in \left(2^{[m]}\right)_{n+1}$, by hypothesis:
\begin{align*}
 \dfrac{w(\{i_0,\dots,i_n\} \cup \{i_{n+1}\})}{w(\{i_0,\dots,i_n\} \cup \{i_{n+1}\} \setminus \{i_0\})} =\dfrac{w(\{i_0,\dots,i_n\})}{w(\{i_0,\dots,i_n\} \setminus \{i_0\})}
 \hspace{2mm} \iff \hspace{2mm}
  \dfrac{w(\{i_0,\dots,i_{n+1}\})}{w(\{i_1,\dots,i_{n+1}\} )} =\dfrac{w(\{i_0,\dots,i_n\})}{w(\{i_1,\dots,i_n)}
\end{align*}
Applying the induction hypothesis, we get
\begin{align*}
w(\{i_0,\dots,i_{n+1}\}) = \dfrac{
\prod_{k=1}^{n+1}w(i_k) \prod_{k=0}^n w(i_k)
}{
\prod_{k=1}^n w(i_k)
}
=\prod_{k=0}^{n+1} w(i_k)
\end{align*}
proving the induction step, and the proof for the theorem.
\end{proof}

\section*{Conclusion}
We introduced  a general framework for weighted sequence complexes, where the weights are derived from  probability distributions  on sequences, and extended the Hodge Laplacian formalism to this setting. 
An important requirement for such a Laplacian is that it has a 
 ``simple'' spectrum in the null models with independent vertices. \medskip

We completely characterized  the spectrum of the associated Laplacian for the independent vertices model, both in the sequence and simplicial complex setting. For  simplicial complexes we found (Corollary  \ref{cor:simplicial}) that the appropriate Laplacians are multiples of the identity. This happens with both natural choices of the weights:  (i) if the weights are the probabilities and (ii) if the weights are the moments, and confirms a  ``common sense'' intuition  that there should not be any {\it preferred}  eigenbasis in the independent model, as the data does not have any ``interesting'' structure beyond the independence of vertices.\\

We found, that unlike in the simplicial complex case, the Laplacian of the independent vertices model for  the sequence complex has a  stereotypical structure (Theorem \ref{bigThm}). Specifically,  it has distinct integer eigenvalues, with  multiplicities that scale exponentially in eigenvalues. Here  the weight function on the sequence complex was chosen to be the conditional probability per equation \eqref{eq:wpsigma}. This choice  is natural in this context, as the length of the sequence is another random variable that does not depend on the assumption that vertices appear in each sequence position independently. The eigenbasis for this independent vertices model provides the analogue of the Fourier transform in the context of the Laplacians in a Euclidean domain. 

There are many open  questions  that stem  from our results. 
For example, both of our independent vertices models can be thought of as  maximal entropy distributions that are constrained by the first moments. What is special about the spectrum of common maximal entropy distributions (such as the Ising model), constrained by the higher moments? What is the interpretation of the diffusion equation  for the Hodge Laplacian on the weighted sequence complex?     What can be said about the probability distributions on sequence complexes that are Markovian? Can one utilize the Hodge Laplacian on sequences for construction of the language models, similar to how the Hodge Laplacian was used in  simplicial neural networks \cite{Ebli2020,Barbarossa_2020,Bodnar2021,roddenberry2021principled,giusti2022simplicial,keros2022dist2cycle,roddenberry2022signal,bodnar2022neural}?   We believe that our results establish a foundation for dimensionality reduction and Fourier analyses of probabilistic models, that are common in theoretical neuroscience and machine-learning.

\bigskip
\noindent {\large \bf  Acknowledgments:}
This work was supported by the NSF Next Generation Networks for Neuroscience Program (award 2014217).

\newpage
\bibliographystyle{unsrt}

\bibliography{refs.bib}

\end{document}